\documentclass[10pt,a4paper]{amsart}%
\usepackage{amsfonts,amsmath,amssymb}
\usepackage{hyperref}
\usepackage{amssymb,amsthm,amsxtra}
\usepackage[usenames]{color}
\usepackage{amscd}
\usepackage{amsthm}
\usepackage{amsfonts}
\usepackage{amssymb}
\usepackage{amsmath}
\usepackage{graphicx}
\usepackage{enumerate}%
\usepackage[all]{xy}
\let\oldtocsubsection=\tocsubsection
\let\oldtocsubsubsection=\tocsubsubsection
\renewcommand{\tocsubsection}[2]{\hspace{1em}\oldtocsubsection{#1}{#2}}
\renewcommand{\tocsubsubsection}[2]{\hspace{2em}\oldtocsubsubsection{#1}{#2}}

\newtheorem*{theorem*}{Theorem}
\newtheorem*{remark*}{Remark}
\newtheorem*{example*}{Example}
\newtheorem{lemma}{Lemma}[subsection]

\newtheorem{remark}[lemma]{Remark}

\newtheorem{theorem}[lemma]{Theorem}
\newtheorem{definition}[lemma]{Definition}
\newtheorem{notation}[lemma]{Notation}

\newtheorem{cond}[lemma]{Condition}

\newtheorem*{conjecture*}{Conjecture}

\newtheorem{thm}[lemma]{Theorem}
\newtheorem{prop}[lemma]{Proposition}
\newtheorem{lem}[lemma]{Lemma}
\newtheorem{defn}[lemma]{Definition}
\newtheorem{notn}[lemma]{Notation}
\newtheorem{cor}[lemma]{Corollary}
\newtheorem{exm}[lemma]{Example}

\newtheorem{conj}[lemma]{Conjecture}
\newtheorem{rem}[lemma]{Remark}

\newtheorem{introtheorem}{Theorem}

\newtheorem{introthm}[introtheorem]{Theorem}

\baselineskip 18pt \textwidth 16cm  \theoremstyle{plain}

\newcommand{\tr}{\operatorname{Tr}}

\newcommand{\Hom}{\operatorname{Hom}}

\newcommand{\diag}{\operatorname{diag}}

\newcommand{\eps}{\varepsilon}

\newcommand{\re}{\operatorname{Re}}
\renewcommand{\Im}{\operatorname{Im}}
\newcommand{\Ker}{\operatorname{Ker}}

\newcommand{\Z}{{\mathbb Z}}

\newcommand{\R}{{\mathbb R}}

\newcommand{\C}{{\mathbb C}}

\newcommand{\Span}{{\operatorname{Span}}}

\newcommand{\bC}{{\mathbb C}}

\newcommand{\alp}{{\alpha}}

\newcommand{\lam}{{\lambda}}

\newcommand{\Fre}{{Fr\'{e}chet \,}}

\newcommand{\cN}{{\mathcal{N}}}

\newcommand{\g}{{\mathfrak{g}}}

\newcommand{\fg}{{\mathfrak{g}}}

\newcommand{\fn}{{\mathfrak{n}}}
\newcommand{\fk}{{\mathfrak{k}}}
\newcommand{\fv}{{\mathfrak{v}}}

\newcommand{\Supp}{\mathrm{Supp}}

\newcommand{\cO}{{\mathcal{O}}}
\newcommand{\GL}{\operatorname{GL}}

\newcommand{\gl}{{\mathfrak{gl}}}

\newcommand{\Sym}{\operatorname{Sym}}
\newcommand{\Spec}{\operatorname{Spec}}

\newcommand{\Gr}{\operatorname{Gr}}
\newcommand{\Coker}{\operatorname{Coker}}

\newcommand{\n}{\mathfrak{n}}

\newcommand{\cM}{\mathcal{M}}

\newcommand{\fp}{\mathfrak{p}}
\newcommand{\fl}{\mathfrak{l}}
\newcommand{\fm}{\mathfrak{m}}
\newcommand{\fq}{\mathfrak{q}}

\newcommand{\fs}{\mathfrak{s}}

\newcommand{\fu}{\mathfrak{u}}

\newcommand{\cV}{\mathcal{V}}

\newcommand{\oE}{\widetilde{E}}
\newcommand{\oPhi}{\widetilde{\Phi}}

\newcommand{\Dima}[1]{{#1}}
\newcommand{\RamiA}[1]{{#1}}
\newcommand{\DimaA}[1]{{#1}}

\newcommand{\p}{{\mathfrak{p}}}
\newcommand{\onto}{{\twoheadrightarrow}}
\newcommand{\U}{{\mathcal{U}}}

\begin{document}

\author{Avraham Aizenbud}
\address{Avraham Aizenbud, Faculty of Mathematics
and Computer Science, The Weizmann Institute of Science POB 26,
Rehovot 76100, ISRAEL.}
\email{aizenr@gmail.com}
\urladdr{\url{http://www.wisdom.weizmann.ac.il/~aizenr/}}

\author{Dmitry Gourevitch}
\address{Dmitry Gourevitch,
Faculty of Mathematics and Computer Science,
Weizmann Institute of Science,
POB 26, Rehovot 76100, Israel }
\email{dmitry.gourevitch@weizmann.ac.il}
\urladdr{\url{http://www.wisdom.weizmann.ac.il/~dimagur}}

\author{Siddhartha Sahi}
\address{Siddhartha Sahi, Department of Mathematics, Rutgers University, Hill Center -
Busch Campus, 110 Frelinghuysen Road Piscataway, NJ 08854-8019, USA}
\email{sahi@math.rugers.edu}

\date{\today}
\title{Derivatives for smooth representations of $GL(n,{\mathbb{R}})$ and $GL
(n,{\mathbb{C}})$}

\begin{abstract}
The notion of derivatives for smooth representations of $GL(n,{\mathbb{Q}}_p)$ was
defined in \cite{BZ-Induced}.
In the archimedean case, an analog of the highest derivative was defined for irreducible
unitary representations in \cite{Sahi-Kirillov} and called the ``adduced" representation. In this paper we define
derivatives of all orders for smooth admissible Fr\'{e}chet representations (of
moderate growth). The real case is more problematic than the p-adic case;
for example arbitrary derivatives need not be admissible. However, the
highest derivative continues being admissible, and for irreducible unitarizable representations coincides with the space of smooth vectors of the adduced representation.

In \cite{AGS2} we prove exactness of the highest derivative functor, and compute highest derivatives of all monomial representations.

We apply those results to finish the computation of adduced representations for all irreducible unitary representations and to prove uniqueness of degenerate Whittaker models for unitary representations, thus completing the results of \cite{Sahi-Kirillov,Sahi-PAMS,SaSt,GS}.
\end{abstract}

\maketitle



%

\section{Introduction}

\setcounter{lemma}{0}

The notion of derivative was first defined in \cite{BZ-Induced} for smooth
representations of $\ G_{n}=GL(n)$ over non-archimedean fields and became a
crucial tool in the study of this category. The purpose of the present paper
is to transfer part of this construction to the archimedean case.

The definition of derivative is based on the \textquotedblleft
mirabolic\textquotedblright\ subgroup $P_{n}$ of $G_{n}$ consisting of
matrices with last row $(0,\dots,0,1)$. The unipotent radical of this subgroup
is an $\left( n-1\right) $-dimensional linear space that we denote $V_{n}$,
and the reductive quotient is $G_{n-1}$. The group $G_{n-1}$ has 2 orbits on
$V_{n}$ and hence also on $V_{n}^{\ast }$: the zero and the non-zero orbit.
The stabilizer in $G_{n-1}$ of a non-trivial character $\psi $ of $V_n$ is isomorphic to $P_{n-1}$.

The construction of derivative in \cite{BZ-Induced} is based on two
functors: $\Phi ^{-}$ and $\Psi ^{-}$. In this paper we denote those functors just by $\Phi$ and $\Psi$. The functor $\Psi$ goes from the
category of smooth representations of the mirabolic group $P_{n}$ to the
category of smooth representations of $G_{n-1}$ (for each $n$) and $\Phi$ goes from the category of smooth representations of $P_{n}$ to the
category of smooth representations of $P_{n-1}$. The functor $\Psi$ is
the functor of (normalized) coinvariants with respect to $V_{n}$ and the functor $\Phi$ is the functor of (normalized) co-equivariants with respect to $(V_{n},\psi )$. The
functor of $k$-th derivative is defined to be the composition $\Psi \circ \Phi^{k-1}$.

Another way to describe those two functors is via the equivalence of
categories of the smooth representations of $P_{n}$ and the category of $%
G_{n-1}$-equivariant sheaves on $V_{n}^{\ast }$. This equivalence is based
on the Fourier transform. Under this equivalence, $\Psi$ becomes the
fiber at $0$ and $\Phi$ becomes the fiber at the point $\psi $. The
functor $\Phi  $ can also be viewed as the composition of two functors:
restriction to the open orbit and an equivalence of categories between
equivariant sheaves on the orbit and representations of the stabilizer of a
point.

In the archimedean case, the notion of fiber of a sheaf behaves differently
than in the non-archimedean case; in particular it is not exact. One way to
deal with this problem is to consider instead the notion of a stalk or a jet. On the
level of representations this means that one uses generalized coinvariants
instead of usual coinvariants. For example, the Casselman-Jacquet functor
is defined in this way. Therefore in our definition we replace the functor $\Psi  $ by the functor of generalized coinvariants. However, we do not
change the definition of the functor $\Phi  $ since we think of it as
restriction to an open set followed by an equivalence of categories, and in
particular it should be exact.

\DimaA{
This gives the following definition of derivative.
Let $\psi_{n}$ be the standard non-degenerate character of $V_{n}$, given by $\psi _{n}(x_1,\dots,x_{n-1}):=\exp(\sqrt{-1}\pi \re x_{n-1}).$ We will also denote by $\psi_n$ the corresponding character of the Lie algebra $\fv_n$.
For all $n$ and for all representations $\pi $ of ${\mathfrak{p}}_{n}$, we define
$$\Phi(\pi):=|\det|^{-1/2} \otimes \pi_{(\fv_n,\psi_n)}:=|\det|^{-1/2} \otimes \pi/\Span\{\alp v - \psi_n(\alp)v \, : \, v \in \pi, \, \alp \in \fv_{n}\}$$
and
$$\Psi(\pi):= \lim_{\overset{\longleftarrow }{l}} \pi /\Span\{\beta v\,|\,v  \in \pi ,\beta \in ({\mathfrak{v}}_{n})^{\otimes l} \}.$$
Now, define $D^k(\pi):=\Psi\Phi^{k-1}(\pi)$.
}
%
%

Consider the case $k=n$; in this case the derivative becomes the (dual) Whittaker
functor. It is well known that the behavior of the Whittaker functor depends
on the category of representations that we consider. For example in the
category of (admissible) Harish-Chandra modules the Whittaker functor gives high
dimensional vector spaces while in the equivalent category of
smooth admissible {Fr\'{e}chet }representations the Whittaker functor gives
vector spaces of dimension $\leq 1$ just as in the non-archimedean case. For
this reason we view the functor $D^{k}$ restricted to the category of smooth
admissible {Fr\'{e}chet} representations as the natural counterpart of the
Bernstein-Zelevinsky derivative.

Nevertheless in order to study this functor we will need to consider also the category of Harish-Chandra modules as well as some other related functors.

In the non-archimedean case the highest non-zero derivative plays a special
role. It has better properties than the other derivatives. In particular in
its definition one can omit the last step of $\Psi  $ since $V_{n-k}$
already acts trivially on the obtained representation. The index of the
highest derivative is called the depth of the representation.
As observed in \cite{GS} the depth can also be described in
terms of the wavefront set of the representation.

In the archimedean case the wavefront set of a representation $\pi$ is
determined by its annihilator variety, and we will use the latter to
define ``depth''. We recall that if $\pi$ is an admissible representation
of $G_{n}$ then its annihilator variety $\mathcal{V}_{\pi}$ is a subset
of the cone of nilpotent $n\times n$ matrices. We define the depth of
$\pi$ to be the smallest index $d$ such that $X^{d}=0$ for all
$X\in \mathcal{V}_{\pi}$.

\begin{example*}
$ $
\begin{enumerate}
\item For a finite dimensional representation $\pi $ of $G_{n}$, $depth(\pi
)=1$, $D^{1}(\pi )=\Phi (\pi )|_{G_{n-1}}=\pi |_{G_{n-1}}$, and $D^{k}(\pi
)=0$ for any $k>1$.

\item $D^{n}=(\Phi) ^{n-1}$ is the Whittaker functor. On the category of
smooth admissible {Fr\'{e}chet} representations it is proven to
be exact in \cite{CHM} and in the category of admissible Harish-Chandra modules in \cite{Kos}.
It is also proven in
\cite{Kos} that $depth(\pi )=n$ if and only if $D^{n}(\pi )\neq 0$ (in both
categories).
\end{enumerate}
\end{example*}

From now on, let $F$ be an archimedean local field and $G_n:=\GL(n,F)$.

In this paper we will mainly be interested in the depth derivative. The following theorem summarizes the main results of this paper.



\begin{introthm}
\label{thm:main} Let $\mathcal{M}_{\infty}(G_{n})$ denote the category of smooth
admissible {Fr\'{e}chet} representations of moderate growth and let $%
\mathcal{M}^{d}_{\infty}(G_{n})$ denote the subcategory of representations of depth $%
\leq d$. Then

\begin{enumerate}
\item \label{mainit:Adm} $D^d$ defines a functor $\cM^d_{\infty}(G_n) \to
\mathcal{M}_{\infty}(G_{n-d})$.

\item \label{mainit:Exact} The functor $D^d : \cM^d_{\infty}(G_n) \to
\mathcal{M}_{\infty}(G_{n-d})$ is exact.

\item \label{mainit:nilp} For any $\pi \in \cM^d_{\infty}(G_n), \,
D^d(\pi)=(\Phi)^{d-1}(\pi)$.

\item \label{mainit:Zero} $D^{k}|_{\mathcal{M}^{d}_{\infty}(G_{n})}=0$ for any $k>d$.

\item \label{mainit:Prod} Let $n=n_1+\cdots +n_d$ and let $\chi_i$ be characters of $G_{n_i}$. Let $\pi= \chi_1 \times \cdots \times \chi_d \in \cM_d(G_n)$ denote the corresponding monomial representation. Then
    $$D^d(\pi)\cong((\chi_1)|_{G_{n_1-1}} \times  \RamiA{\cdots}  \times (\chi_d)|_{G_{n_d-1}}) $$
\Dima{
\item \label{mainit:A} If $\tau$ is an irreducible unitary representation of $G_n$ and $\tau^{\infty}$ has depth $d$ then $D^d(\tau^{\infty})\cong (A\tau)^{\infty}$, where $A\tau$ denotes the adduced representation defined in \cite{Sahi-Kirillov} (see \S\ref{subsubsec:PrelA}).
    }
\end{enumerate}
\end{introthm}

\begin{remark*}$ $
\begin{enumerate}[(i)]
\item Part (\ref{mainit:nilp}) of the theorem means that ${\mathfrak{v}}
_{n-d+1}$ acts nilpotently on $(\Phi  )^{d-1}(\pi )$. Unlike the p-adic
case, $V_{n-d+1}$ need not act trivially on $(\Phi  )^{d-1}(\pi )$.

\item The proofs of parts (\ref{mainit:Exact}),(\ref{mainit:Zero}) are  based on the results of \cite{AGS2}

\item In this paper we do not prove that for $\pi$ of depth $d$, $D^{d}(\pi )\neq 0$. This is in fact true but needs an additional argument which is provided in \cite{GS-Gen}. However, for $\pi$ monomial or unitarizable it follows from parts \eqref{mainit:Prod} and \eqref{mainit:A} of Theorem \ref{thm:main} respectively.

\item We prove analogs of items (\ref{mainit:Adm}), (\ref{mainit:nilp}), and
(\ref{mainit:Zero}) of the theorem also for the category of Harish-Chandra modules.

\end{enumerate}
\end{remark*}

\subsection{Related results}
$\quad$\\
As was mentioned earlier, the non-archimedean counterpart of this paper was done in \cite{BZ-Induced}. In the archimedean case, an analogous notion to the notion of highest derivative was introduced for irreducible unitary representations in \cite{Sahi-Kirillov} and called ``adduced representation".

The case of smooth representations over archimedean fields differs from the above cases in several ways. First of all, we do not have a suitable category of representations of $P_n$. The existence of such a category in other cases was crucial for the study of derivatives.

Another difference is the relation between the derivative and the classification of irreducible representations. In the non-archimedean case, the theory of derivatives was the base for the Zelevinsky classification.
In the unitary case, the notion of adduced representation is closely related with the Vogan classification (see \S \ref{subsubsec:IntroAddRep} and \S\ref{subsubsec:PrelA}).

In our case, we do not currently have a classification that is suitable for the theory of derivatives. The Langlands classification is not compatible with the notion of derivative. In particular, it is hard to read from the  Langlands classification the annihilator variety or even the 
depth of the representation, which are crucial notions in the study of derivatives. We hope that eventually it will be possible to make an archimedean analog of the Zelevinsky classification. However, it seems to be quite difficult. Let us explain why.

The Langlands classification presents any irreducible representation as a ``smallest" subquotient of a parabolic induction of a discrete series representation. In the non-archimedean case the discrete series representations can be presented as ``largest" subquotients of parabolic induction of cuspidal representations.
The Zelevinsky classification is dual (under the Zelevinsky involution) to the Langlands classification. Namely, the Zelevinsky classification presents any irreducible representation as a ``largest" subquotient of a generalized Speh representation corresponding to a segment of cuspidal representations and any such Speh representation as a ``smallest" subquotient of a parabolic induction of a cuspidal representation.

Such a nice picture cannot exist in the archimedean case or even in the complex case.
Indeed, $GL_n(\C)$ has discrete series representations only for $n=1$. Thus one would expect that in the complex case the natural analog of generalized Speh representation as above exists only for $GL_1(\C)$.
Therefore, the naive analogy would suggest that any irreducible representation is the ``largest" subquotient of a principal series representation. This is not true. Moreover, it is even not true that  any irreducible representation is the ``largest" subquotient of a monomial representation (i.e. a Bernstein-Zelevinsky product of characters)\Dima{, or even the ``largest" subquotient of a BZ-product of finite-dimensional representations}.

The $n$-th derivative of representations of $G_n$ is the Whittaker functor. Thus, a special case of Theorem \ref{thm:main} implies that the Whittaker functor is exact and maps a principal series representation to a one-dimensional space. This is known for any quasi-split reductive group by \cite{Kos} and \cite{CHM}.

\subsection{Structure of our proof}
$ $\\
We start working in the Harish-Chandra category. We show that for a Harish-Chandra module $\pi$ of depth $d$, $D^d(\pi)$ is an admissible
Harish-Chandra module over $G_{n-d}$. From this we deduce that $D^d(\pi)=\Phi^{d-1}(\pi)$ and $D^k(\pi)=0$ for any $k>d$.

 In \cite{AGS2} we analyze the functor $\Phi^k$ as a functor from $\mathcal{M}_{\infty}(G_n)$
to the category of representations of ${\mathfrak{p}}_{n-k}$. We prove that
it is exact and for any $\pi \in \mathcal{M}_{\infty}(G_n),$ $\Phi^k(\pi)$ is a
Hausdorff space. This means that $\mathfrak{u}_n^k( \pi \otimes \psi^k)$ (see Notation \ref{not:main}) is
a closed subspace. In fact, we prove those statements for a wider class of representations of $\fp_n$.

Then we deduce items \eqref{mainit:Adm}-\eqref{mainit:Zero} of Theorem \ref{thm:main} from the above results.

In \cite{AGS2} We prove \eqref{mainit:Prod} by computing $\Phi^{d-1}$ on certain representations of $\fp_n$ using the results on exactness and Hausdorffness of $\Phi$ for those representations.

Finally, we prove \eqref{mainit:A} using \eqref{mainit:Adm}-\eqref{mainit:Prod}, \cite{GS}, and the Vogan classification.

\subsubsection{Admissibility}
Let ${\mathfrak{n}}_n$ denote the Lie algebra of upper triangular nilpotent $%
n \times n$ matrices. A finitely-generated $(\g_n,K)$-module
 $\pi$ is admissible if and
only if it is finitely generated over $\mathfrak{n}_n$ (see Theorem \ref{thm:EqAdm}). Thus, we know that $\Phi^{d-1}(\pi)$ is finitely generated over $\mathfrak{n}_{n-d+1}$ and we need to show that it is in fact finitely
generated over $\mathfrak{n}_{n-d}$.

To do that we use two invariants of modules over Lie algebras: annihilator
variety (see \ref{subsubsec:AnnVar}) and associated variety (see \ref{subsec:Filt}). Both are analogs of the notion of support of a module over a commutative algebra. Both are subvarieties of the dual space to the Lie
algebra, and the annihilator variety includes the associated variety. The
definition of the associated variety requires the module to be filtered, but
the resulting variety does not depend on the choice of a good filtration on
the module.

To prove that $\Phi^{d-1}(\pi )$ is finitely generated over $\mathfrak{n}_{n-d}$ we show that the associated variety of $\Phi^{d-1}(\pi )$, viewed as
a module over $\mathfrak{n}_{n-d+1}$, is included in $\mathfrak{n}_{n-d}^{\ast }$. Using a lemma that we prove in \S \ref{subsec:keylem},
we deduce this from the bound on the annihilator variety of $\pi $ that we
have by definition of the depth of $\pi $.

\subsection{Applications}\label{subsec:Appl}

\subsubsection{Degenerate Whittaker models}
Let $N_n<G_n$ be the nilradical of a Borel subgroup. Let $\chi$ be any unitary character of $N_n$.
A (degenerate) Whittaker functional on a smooth representation of $G_n$ is an $(N_n,\chi)$-equivariant functional. Such functionals were studied in \cite{GS}. In particular, \cite{GS} associates a character $\chi_\pi$ of $N_n$ to any irreducible representation $\pi$, using the annihilator variety of $\pi$ and proves existence of the corresponding degenerate Whittaker functionals for unitarizable $\pi$. In this paper we deduce uniqueness of those functionals from Theorem \ref{thm:main}.

\subsubsection{Adduced representations}\label{subsubsec:IntroAddRep}
In \cite{Sahi-Kirillov,Sahi-PAMS,SaSt} adduced representations were computed for characters, Stein complementary series and Speh representations. Moreover it was proven that adduced representation of a Bernstein-Zelevinsky product is the product of the corresponding adduced representations. Thus, by Vogan classification, the task of computing adduced representations was reduced to the case of Speh complementary series. In this paper we perform this computation, based on Theorem \ref{thm:main}.

%
\subsubsection{Intertwining Operators}
In \cite{IntOp} the second derivative $E^2$ is used in order to show that the space of intertwining operators between two representations of $G_n$ each on spaces of global sections of equivariant line bundles on Grassmanians is at most one-dimensional, and to find all the cases when it is non-zero.

\subsubsection{Future applications}\label{subsubsec:FutApp}
In \cite{CPS},  Cogdell and Piatetski-Shapiro use Bernstein-Zelevinsky derivatives to compute local Rankin-Selberg integrals for $GL_n \times GL_m$ over p-adic fields. Hopefully, the derivatives defined in this paper can be used to compute Rankin-Selberg integrals at the archimedean place in a similar way. J. Cogdell has informed us that this is being investigated by his student J. Chai in his PhD thesis at Ohio State University.

\subsection{Tools developed in this paper: a bound on annihilator variety in terms of coinvariants }
%
%
%
In order to prove the admissibility of the depth's derivative we need to
show a certain bound on its associated variety. For this we analyze the
following situation: Let ${\mathfrak{h}}_1 \vartriangleleft {\mathfrak{h}}%
_2< {\mathfrak{g}} $ be Lie algebras. Let $\psi$ be a character of ${%
\mathfrak{h}}_1$ that is stabilized by ${\mathfrak{h}}_2$ and let $\pi$ be a
filtered ${\mathfrak{g}}$-module. We have certain bounds on the annihilator
variety of $\pi$ in ${\mathfrak{g}}^*$ and we are interested in the associated
variety of $\pi_{{\mathfrak{h}}_2,\psi}$ in $({\mathfrak{h}}_1/{\mathfrak{h}}_2)^*.$

In \S \ref{subsec:keylem} we provide some bounds on this associated
variety under certain technical assumptions.

\subsection{Structure of the paper}$ $\\
In \S \ref{sec:Prel} we give the necessary preliminaries on Harish-Chandra modules, admissible smooth representations, annihilator varieties and Bernstein-Zelevinsky product.

In \S \ref{sec:Main} we formulate the main results and prove some implications between them. We reduce Theorem \ref{thm:main} to admissibility of the depth derivative in the Harish-Chandra category.

In \S \ref{sec:App} we discuss the relation of the notion of depth derivative to the notion of adduced representation and deduce the applications discussed in \S \ref{subsec:Appl}. We also provide the necessary preliminaries on the Vogan classification, adduced representations and degenerate Whittaker models.

In \S \ref{sec:Adm} we prove admissibility of the depth derivative in the Harish-Chandra category. This section is purely algebraic. In \ref{subsec:AdmPfStr} we give an overview of the proof. In \S \ref{subsec:Filt} we give some preliminaries on filtrations and associated varieties. In \S \ref{bigraded}-\ref{subsec:keylem} we formulate and prove a  key lemma) that connects the annihilator variety of a representation $\pi$ to the associated variety of co-equivariants of $\pi$ with respect to a certain subalgebra and its character. In \S \ref{subsec:PfHCAdm} we deduce the admissibility from the key lemma.

In Appendix \ref{sec:PfWonFilt} we present a proof of Proposition \ref{prop:WonFilt} which is an important tool for using filtrations on $\g_n$-modules. This proposition was proven by O. Gabber and written in the \Dima{unpublished lecture notes} \cite{Jos81}. We present a proof here for the sake of completeness.

\subsection{Acknowledgements}

We cordially thank Semyon
Alesker, Joseph Bernstein, Jim Cogdell, Laurent Clozel, Maria Gorelik, Anthony Joseph, Peter Trapa and David Vogan for fruitful discussions, and the referee for very careful proofreading and for his remarks.

\Dima{Part of the work on this paper was done during the program ``Analysis on Lie Groups" at the Max Planck Institute for Mathematics (MPIM) in Bonn.
The authors would like to thank the organizers of the programm, Bernhard Kroetz,  Eitan Sayag and Henrik Schlichtkrull, and the administration of the MPIM for their hospitality.

A.A. was partially supported by NSF grant DMS-1100943 and ISF grant 687/13;

D.G. was partially supported by ISF grant 756/12 and a Minerva foundation grant.
}

\section{Preliminaries} \label{sec:Prel}

\subsection{Notation and conventions}\label{subsec:PrelNot}

\begin{itemize}
\item We will denote real algebraic groups by capital Latin letters and their complexified Lie algebras by small Gothic letters.

\item Let $\g$ be a complex Lie algebra. We denote by $\cM(\g)$ the category of (arbitrary) $\g$-modules. Let $\psi$ be a character of $\g$.
For a module $M \in \cM(\g)$ denote by $M^\g$ the space of $\g$-invariants, by $M^{\g,\psi}$ the space of $(\g, \psi)$-equivariants, by $M_{\g}$ the space of coinvariants, i.e. $M_{\g}:=M/\g M$ and by $M_{\g,\psi}$ the space of $(\g,\psi)$-coequivariants, i.e. $M_{\g,\psi}:=(M \otimes (-\psi))_{\g}$.
\item We also denote by $M_{gen, \g}$ the space of the generalized co-invariants, i.e. %
$$M_{gen, \g}=\lim_{\overset{\longleftarrow}{l}} M/\Span(\{\alp v\, |\, v \in M, \alp \in (\g)^{\otimes l}\}).$$

\item By a composition of $n$ we mean a tuple $(n_1 , \dots ,  n_k)$ of natural numbers such that $n_1+\cdots+n_k=n$. By a partition we mean a composition which is non-increasing, i.e. $n_1 \geq \cdots \geq n_k$.

\item For a composition $\lambda=(n_1 , \dots ,  n_k)$ of $n$ we denote by $P_{\lam}$ the corresponding block-upper triangular parabolic subgroup of $G_n$.
\Dima{For example, $P_{(1,\dots,1)}=B_n$ denotes the standard Borel subgroup, $P_{(n)}=G_n$ and $P_{(n-1,1)}$ denotes the standard maximal parabolic subgroup that includes $P_n$.}
\end{itemize}

\subsection{Harish-Chandra modules and smooth representations}\label{subsec:HC}

In this subsection we fix a real reductive group $G$, a minimal parabolic subgroup of $P \subset G$, and let $N$ denote the nilradical of $P$.
We also fix a maximal compact subgroup $K \subset G$. Let $\fg,\fn,\fk$ denote the complexified Lie algebras of $G,N,K$, and let $Z_{G}:=U({\mathfrak{g}})^{G}$.

\begin{defn}
A $(\g,K)$-module is a  $\g$-module $\pi$ with a locally finite action of $K$ such the two induced actions of $\fk$ coincide and
$\pi(ad(k)(X))=\pi(k)\pi(X)\pi(k^{-1})$ for any $k\in K$ and $X \in \g$.

A finitely-generated $(\g,K)$-module is called admissible if any representation of $K$ appears in it with finite (or zero) multiplicity. In this case we also call it a Harish-Chandra module.
\end{defn}

\begin{thm}[Harish-Chandra, Osborne, Stafford, Wallach] \label{thm:EqAdm}
Let $\pi$ be a finitely generated $(\g,K)$-module. Then the following properties of $\pi$ are equivalent.
\begin{enumerate}
\item \label{Eqit:Adm} $\pi$ is admissible.
\item \label{Eqit:FinLen} $\pi$ has finite length.
\item \label{Eqit:ZFin} $\pi$ is $Z_G$-finite.
\item \label{Eqit:NFinGen} $\pi$ is finitely generated over $\n$.
\end{enumerate}
\end{thm}
The implications (\ref{Eqit:Adm}) $\Rightarrow$ (\ref{Eqit:FinLen}) and
(\ref{Eqit:NFinGen}) $\Rightarrow$ (\ref{Eqit:Adm})  are proven using the Casselman-Jacquet functor, see \cite[\S 4.2]{Wal1}. The implication
(\ref{Eqit:FinLen}) $\Rightarrow$ (\ref{Eqit:ZFin}) follows from Schur's Lemma, and the implication (\ref{Eqit:ZFin})  $\Rightarrow$  (\ref{Eqit:NFinGen}) is proven in \cite[ \S 3.7]{Wal1}
\begin{notation}
\RamiA{
For a real reductive group  $G$ we denote by $K$ its maximal compact subgroup, by $\g$ its complexified Lie algebra, and
 by $\cM(\g)$ the category of $\g$-modules. Denote by $\cM_{\infty}(G)$ the category of smooth admissible \Fre representations of $G$ of moderate growth (see \cite[\S11.5]{Wal2} or \cite{CasGlob}), where admissible means that the space of $K$-finite vectors is an admissible $(\g,K)$-module. Denote by  $\cM_{HC}(G)$ the category of admissible $(\g,K)$-modules. Note that both $\cM_{\infty}(G)$ and $\cM_{HC}(G)$ are naturally subcategories of $\cM(\g)$. We denote by $HC:\cM_\infty(G) \to \cM_{HC}(G)$ the functor of $K$-finite vectors.}
\end{notation}

\begin{thm}[Casselman-Wallach, see \cite{Wal2}, \S 11.6.8] \label{thm:CW}
The functor $HC:\cM_{\infty}(G) \to \cM_{HC}(G)$ is an equivalence of categories.
\end{thm}

In fact, Casselman and Wallach construct an inverse functor $\Gamma: \cM_{HC}(G) \to \cM_{\infty}(G)$, that is called the Casselman-Wallach globalization functor, or the  Casselman-Wallach canonical completion (see \cite[Chapter 11]{Wal2} or \cite{CasGlob} or, for a different approach, \cite{BerKr}).

\begin{cor} \label{cor:CW}
$ $
\begin{enumerate}[(i)]
\item  \label{it:Ab} The category $\cM_{\infty}(G)$ is abelian.
\item \label{it:ClosIm} Any morphism in $\cM_{\infty}(G)$ has closed image.
\end{enumerate}
\end{cor}
\begin{proof}
(\ref{it:Ab}) The category $\cM_{HC}(G)$ is clearly  abelian and by the theorem it is equivalent to $\cM_{\infty}(G)$.

(\ref{it:ClosIm}) Let $\phi:\pi \to \tau$ be a morphism in $\cM_{\infty}(G)$.
Let $\tau' = \overline{\Im {\phi}}$, $\pi' = {\pi/\ker {\phi}}$ and $\phi':\pi' \to \tau'$ be the natural morphism. Clearly $\phi'$ is monomorphic and epimorphic in the category $\cM_{\infty}(G)$. Thus by (\ref{it:Ab}) it is an isomorphism. On the other hand, $\Im \phi' = \Im \phi \subset \overline{\Im \phi} = \tau'$. Thus $\Im {\phi} = \overline{\Im \phi}$.
\end{proof}


\subsection{The annihilator variety and associated partition} \label{subsubsec:AnnVar}

For an associative algebra $A$ the annihilator of a module $\left(
\tau,W\right)  $ is
\[
Ann(\tau)=\left\{  a\in A:\tau\left(  a\right)  w=0\text{ for all }w\in
W\right\}.
\]
If $A$ is abelian then the support of $\tau$ is defined to be the variety corresponding to the ideal $Ann(\tau)$, i.e. $\text{Zeroes}\left(  Ann(\tau)\right)  $.

If $\left(  \tau,W\right)  $ is a module for a Lie algebra $\mathfrak{g}$,
then one can apply the above considerations to the enveloping algebra
$U\left(  \mathfrak{g}\right)  $. While $U\left(  \mathfrak{g}\right)  $ is
not abelian it admits a natural filtration $U^n$ such that $gr\left(  U\left(
\mathfrak{g}\right)  \right)  $ is the symmetric algebra $\Sym\left(
\mathfrak{g}\right)  ,$ and \RamiA{hence one has a symbol} map 
\Dima{$\sigma$}
from $U\left(
\mathfrak{g}\right)  $ to $\Sym\left(  \mathfrak{g}\right)  $. We let $gr\left(
Ann(\tau)\right)  $ be the ideal in $\Sym\left(  \mathfrak{g}\right)  $
generated by the symbols $\left\{  \sigma\left(  a\right)  \mid a\in Ann\left(
\tau\right)  \right\}  $ and define the annihilator variety of $\tau$ to
be
\[
\mathcal{V}\left(  \tau\right)  =\text{Zeroes}\left(  gr\left(  Ann(\tau)\right)
\right)  \subset\mathfrak{g}^{\ast}%
\]

If $\mathfrak{g}$ is a complex reductive Lie algebra and $M$ is an irreducible $\mathfrak{g}$-module, then it was shown
by Borho-Brylinski (see \cite{BB1}) and Joseph (see \cite{Jos85}) that $\mathcal{V}(M)$ is
the closure $\overline{{\mathcal{O}}}$ of a single nilpotent coadjoint orbit
${\mathcal{O}}$, that we call the associated orbit of $M$.
If $G$ is a reductive group and $\pi \in \cM_{\infty}(G)$ is admissible  then $\pi^{HC}$ is dense in $\pi$ and since the action of $\g$ on $\pi$ is continuous we get $\cV(\pi) = \cV(\pi^{HC})$.


\subsubsection{The case of $G_n$}
\Dima{Suppose first that $F=\R$. Then} $\g_n=\gl(n,\bC)$ and \RamiA{ we identify $\g_n$ and $\g_n^*$ with the space $n\times n$ complex matrices, in the usual manner}. By Jordan's theorem, nilpotent orbits in $\g_n^*$ are given by partitions of $n$, i.e. tuples $(n_1 , \dots ,  n_k)$ such that $n_1 \geq \cdots \geq n_k$ and $n_1 + \cdots + n_k =n$ (see \cite[Proposition 3.1.7]{CoMG}).
For a partition $\lambda$ we denote by $\cO_{\lambda}$ the corresponding nilpotent coadjoint orbit. We sometimes use exponential notation for partitions; thus $4^2 2^1 1^3$ denotes $(4, 4, 2, 1, 1, 1)$.

If $\pi \in \cM_{\infty}(G_n)$ is irreducible and $\lambda$ is the partition of $n$ such that $\cV(\pi)=\overline{\cO_{\lambda}}$ we call $\lambda$ the \emph{associated partition of} $\pi$ and denote $\lambda=AP(\pi)$.
For example, if $\pi$ is finite-dimensional then $\cV(\pi)=\{0\}$ and $AP(\pi)= 1^n$ and if $\pi$ is generic then, by \cite{Kos},
$\cV(\pi)$ is the nilpotent cone of $\g_n^*$ and $AP(\pi)=n^1$.

Let us introduce the following definition of depth.

\begin{defn}
Let $\tau \in \cM(g_n)$. Define $depth(\tau)$ to be the smallest number $d$ such that $A^d=0$ for any \RamiA{$A \in \cV(\tau)$}. Denote by
$\cM_{HC}^d(G_n)$ and $\cM_{\infty}^d(G_n)$ the subcategories of $\cM_{HC}(G_n)$ and $\cM_{\infty}(G_n)$ consisting of representations of depth at most $d$.
\end{defn}

It is easy to see that for an irreducible representation $\pi$ with associated partition $(n_1,..,n_k), \, depth(\pi)=n_1$ and the depth of an extension of two representations is the maximum of their depths.

Let us now consider the case $F=\C$. Then $\g_n=\gl(n,\bC)\oplus \gl(n,\bC)$ and coadjoint nilpotent orbits are given by pairs of partitions. However, if $\pi \in \cM_{\infty}(G_n)$ is irreducible then the maximal orbit in $\cV(\pi)$ is symmetric and thus corresponds to a single partition that we call the associated partition. For any $\pi \in \cM_{HC}(GL_n(\C))$ we define $depth(\pi)$ to be the smallest number $d$ such that $A^d=B^d=0$ for any \RamiA{$(A,B) \in \cV(\pi)$}.

\subsection{Parabolic induction and Bernstein-Zelevinsky product} \label{subsec:ParInd}

Let $G$ be a real reductive group, $P$ be a parabolic subgroup, $M$ be the Levi quotient of $P$ and $pr:P \to M$ denote the natural map.

\begin{notn}
$ $
\begin{itemize}
\item For a Lie group $H$ we denote by $\Delta_H$ the modulus character of $H$, i.e. the absolute value of the determinant of the infinitesimal adjoint action.
\item For $\pi \in \cM_{\infty}(M)$ we denote by $I_P^G(\pi)$ the normalized parabolic induction of $\pi$, i.e. the space of smooth functions $f :G \to \pi$ such that $f(pg) = \Delta_{P}(p)^{1/2} \pi(pr(p))f(g)$, with the action of $G$ given by $(I_P^G(\pi)(g)f)(x):=f(xg)$.
\end{itemize}
\end{notn}

The behavior of the annihilator variety under parabolic induction is described by the following theorem.

\begin{thm}\label{thm:AnnVarProd}
Note that we have a natural embedding $\fm^* \hookrightarrow \fp^*$ and a natural projection $r:\g^* \to \fp^*$. Let $\pi \in \cM_{\infty}(M)$. Then $\cV(I_P^G(\pi)) = G_{\C}\cdot r^{-1}(\cV(\pi))$, where $G_{\C}$ is the complexification of $G$.
\end{thm}
This theorem is well-known and can be deduced from \cite[Theorem 2]{BB2}.

\subsubsection{Bernstein-Zelevinsky product}

We now introduce the Bernstein-Zelevinsky product notation for parabolic induction.

\begin{definition}\label{def:BZProd}
If $\alpha=(n_{1} , \dots ,    n_{k})$ is a composition of $n$ and
$\pi_{i}\in\cM_{\infty}(G_{n_{i}})$ then $\pi_{1}\otimes\cdots\otimes\pi_{k}$ is
a representation of $L_{\alpha}\approx G_{{\alpha}_{1}
}\times\cdots\times G_{{\alpha}_{k}}$. We define
\[
\pi_{1}\times\cdots\times\pi_{k}=I_{P_{\alpha}}^{G_{n}}\left(  \pi_{1}\otimes\cdots\otimes\pi_{k}\right)
\]
\end{definition}

\RamiA{ $\pi_{1}\times\cdots\times\pi_{k}$ will be referred to below as the Bernstein-Zelevinsky product, or the BZ-product, or sometimes just the product of $\pi_{1},\dots,\pi_{k}$. It is well known (see e.g. \cite[\S 12.1]{Wal2}) that the product is commutative in the Grothendieck group.}
From Theorem \ref{thm:AnnVarProd} we obtain

\begin{cor}\label{cor:DepthProd}
Let $\pi_1 \in G_{n_1}$ and $\pi_2 \in G_{n_2}$. Then $depth(\pi_1 \times \pi_2) = depth(\pi_1)+depth(\pi_2)$.
\end{cor}

\section{Main results}\label{sec:Main}
\setcounter{lemma}{0}

\begin{notation}\label{not:main}
$ $
\begin{itemize}
\item Fix $F$ to be either $\R$ or $\C$.
\item $G_n :=GL(n,F)$, we embed $G_n \subset G_m$ for any $m>n$ by sending any $g$ into a block-diagonal matrix consisting of $g$ and $Id_{m-n}$. We denote the union of all $G_n$ by $G_{\infty}$ and all the groups we will consider will be embedded into $G_{\infty}$ in a standard way.
\item We denote by $P_{n}\subset G_{n}$ the mirabolic subgroup (consisting of matrices with last row $(0 , \dots ,  0,1)$).
\item Let $V_n \subset P_n$ be the unipotent radical. Note that $V_n \cong F^{n-1}$ and $P_n = G_{n-1} \ltimes V_n$. Let $U_n^k := V_{n-k+1} V_{n-k+2} \cdots V_{n}$ and $S_n^k := G_{n-k} U_n^k$. Note that $U_n^k$ is the unipotent radical of $S_n^k$. Let $N_n:=U_n^n$.
\item Fix a non-trivial unitary additive character $\theta$ of $F$, given by $\theta(x)=\exp(\sqrt{-1}\pi \re x)$.
\item Let $\bar{\psi}_n^k:U_n^k \to F$ be the standard non-degenerate homomorphism\Dima{, given by $\bar{\psi}_n^k(u)=\sum_{j=n-k}^{n-1} u_{j,j+1}$} and let $\psi_n^k:=\theta \circ \bar{\psi}_n^k$.
\end{itemize}

We will usually omit the $n$ from the notations $U_n^k$ and $S_n^k$, and both indexes from $\psi_n^k$.
\end{notation}

\begin{defn}\label{def:main}
\DimaA{Define functors $\Phi: \cM(\fp_n) \to \cM(\fp_{n-1})$ by $\Phi(\pi):=\pi_{\fv_n,\psi} \otimes |\det|^{-1/2}$ and $\Psi,\Psi_0: \cM(\fp_n) \to \cM(\fg_{n-1})$ by $\Psi(\pi):= \pi_{gen,\fv_{n}}$ and $\Psi_0(\pi):= \pi_{\fv_{n}}$. }

For a $\p_n$-module $\pi$ we define three notions of derivative:
\DimaA{
\begin{enumerate}
\item $E^k(\pi):=\Phi^{k-1}(\pi):=\pi_{\fu^{k-1},\psi^{k-1}}\otimes |\det|^{-(k-1)/2}.$ Clearly it has a structure of a $\p_{n-k+1}$ - representation.

\item $D^k(\pi):= \Psi(E^k(\pi)).$

\item $B^k(\pi):= \Psi_0(E^k(\pi)).$
\end{enumerate}
}

Note that the derivative functor $D^k$ was defined in the introduction.
For convenience we will also use untwisted versions of the above functors, defined by \Dima{$\oPhi(\pi):=\Phi(\pi) \otimes |\det|^{1/2}$, and $\oE^k(\pi):=E^k(\pi) \otimes |\det|^{(k-1)/2}$.}

We denote the restrictions of the above functors to the subcategory $\cM_{\infty}(G_n)$ by $B_{\infty}^k$, $D_{\infty}^k$, and $E_{\infty}^k$. Similarly, we denote the restrictions to $\cM_{HC}(G_n)$ by $B_{HC}^k$, $D_{HC}^k$ and $E_{HC}^k$. Note that if $\pi \in \cM_{\infty}(G_n)$ then $D_{\infty}^k(\pi)$ has a natural structure of a $P_{n-k+1}$ topological representation and  if $\pi \in \cM_{HC}(G_n)$ the $D_{HC}^k(\pi)$ has a natural structure of a $K'$ representation where $K'$ is the maximal compact subgroup of $G_{n-k}$. The same is true for the functors $B$ and $E$.

We have natural maps: $E^k \to D^k \to B^k$,
$ HC \circ B^k_{\infty} \to B^k_{HC} \circ HC$, $HC \circ D^k_{\infty} \to D^k_{HC}\circ HC$ and $HC \circ E^k_{\infty} \to E^k_{HC}\circ HC$. Here $HC$ is the functor of taking $K-$ finite vectors and the last three maps are maps of $K$ representations and $\p$ representations.
\end{defn}

\begin{prop}\label{prop:D3Adm}
Let $\pi \in \cM_{HC}(G_n)$. Then $B_{HC}^k(\pi)$ is admissible for any $1\leq k \leq n$.
\end{prop}
\begin{proof}
By Theorem \ref{thm:EqAdm}, $\pi$ is finitely generated over $\n_n$.
Note that the functor $B_{HC}^k$ quotients by the last $k$ columns of $\n_n$ (with an appropriate character) and thus $B_{HC}^k(\pi)$ is finitely generated over $\n_{n-k}$.
 Therefore, by Theorem \ref{thm:EqAdm} again, $B_{HC}^k(\pi)$ is admissible.
\end{proof}

\begin{remark}
Suppose $F=\R$. Then the center of $\U(\g_n)$ is isomorphic to the algebra of symmetric polynomials in $n$-invariants, and its characters are given by multisets of size $n$ in $\C$.  Let  $\pi \in \cM_{HC}(G_n)$, and let $S^{\prime}$ be the multiset
corresponding to an infinitesimal character of $B^k(\pi)$.
 Then $S'$ is obtained from
the multiset corresponding to some infinitesimal character of $\pi$ by deleting $k$ of the elements and adding $1/2$ to
each of the remaining ones. This is proven by the argument in the proof of \cite[Proposition 4.5.4]{GS}. A similar statement holds for $F=\C$.
\end{remark}

In \RamiA{\S  \ref{sec:Adm}} we prove the following theorem
\begin{theorem}\label{thm:HCAdm}
Let $\pi \in \cM_{HC}^d(G_n)$. Then the restriction of $E_{HC}^d(\pi)$ to $\g_{n-d}$ is admissible.
\end{theorem}

\begin{cor}\label{cor:HCNilp}
Let $\pi \in \cM_{HC}^d(G_n)$. Then $\fv_{n-d+1}$ acts nilpotently on $E_{HC}^d(\pi)$. Namely, there exists a number $k$ such that for any $X \in \fv_{n-d+1}$, $X^k$ acts by zero on $E_{HC}^d(\pi)$.
\end{cor}
\begin{proof}
Let $\tau:= E_{HC}^d(\pi)$. Since it is admissible over $\g_{n-d}$, it is finite over the center of $U(\g_{n-d})$. Hence there exists a polynomial $p$ such that $\tau(p(I))=0$, where $I \in \g_{n-d}$ denotes the identity matrix. Let $k$ be the degree of $p$ and $X \in \fv_{n-d+1}$ be any element. We will show that $\tau(X)^k=0$.

Note that $[I,X] = X$ and hence $ad(X)^k(I^k)= k!(-X)^k$ and $ad(X)^kI^{k-i}=0$ for any $i>0$. Thus $ad(X)^k (p(I))$ is proportional to $X^k$. On the other hand, $\tau(p(I))=0$, hence $\tau(ad(X)^k (p(I)))=0$ and thus $\tau(X)^k=0$.
\end{proof}

\begin{cor}
Let $\pi \in \cM_{HC}^d(G_n)$. Then
\begin{enumerate}
\item $D_{HC}^{d}(\pi)=E_{HC}^{d}(\pi).$
\item $E_{HC}^{d+1}(\pi)=D_{HC}^{d+1}(\pi)=B_{HC}^{d+1}(\pi)=0$.
\end{enumerate}
\end{cor}

\begin{thm}[\cite{AGS2}, Theorem A] \label{thm:ExactHaus}

For any $0 <k \leq n$
\begin{enumerate}
\item \label{it:Exact}
 $E_{\infty}^{k}$ is an exact functor
\item \label{it:Haus} For any $\pi \in \cM_{\infty}(G_n)$, the natural (quotient) topology on $E_{\infty}^{k}(\pi)$ is Hausdorff, i.e. $\fu^k(\pi \otimes (-\psi^k))$ is closed in $\pi$.
\end{enumerate}
\end{thm}

\begin{cor} \label{cor:SmoothDer}
Let $\pi \in \cM_{\infty}(G_n)$ be of depth $d$. Let $0 <k \leq n$. Then
\begin{enumerate}
\item The natural map $p:E_{HC}^{k}(HC(\pi)) \to HC(E_{\infty}^{k}(\pi))$ is onto. \label{it:onto}
\item $E_{\infty}^{d}(\pi)= D_{\infty}^{d}(\pi) \in \cM_{\infty}(G_{n-d})$. \label{it:1=2}
\item $E_{\infty}^{d+1}(\pi)=D_{\infty}^{d+1}(\pi)=B_{\infty}^{d+1}(\pi)= 0$. \label{it:d+1=0}
\end{enumerate}
\end{cor}

For the proof we will need the following standard lemma.

\begin{lem}\label{lem:FinDense}
If a locally convex Hausdorff topological  vector space $W$ has a dense finite dimensional subspace then $W$ is finite dimensional.
\end{lem}


\begin{proof}[Proof of Corollary \ref{cor:SmoothDer}]
Let us first prove part \eqref{it:onto}.
The quotient map $\pi \to E_{\infty}^{k}(\pi)$ is onto. Thus, $HC(\pi) \to E_{\infty}^{k}(\pi)$ has dense image, hence $E_{HC}^{k}(HC(\pi))\to E_{\infty}^{k}(\pi)$ has dense image and hence $p$ has dense image. Let $\rho$ be a $K_{n-k}$-type. Consider $p^{\rho}: (E_{HC}^{k}(HC(\pi)))^{\rho} \to (HC(E_{\infty}^{k}(\pi)))^{\rho}$. It must also have dense image. By Theorem \ref{thm:HCAdm}, $ (E_{HC}^{k}(HC(\pi)))^{\rho}$ is finite dimensional and by Theorem \ref{thm:ExactHaus}, $(HC(E_{\infty}^{k}(\pi)))^{\rho}$ is Hausdorff. Thus, $p^{\rho}$ is onto for any $\rho$ and hence, by Lemma \ref{lem:FinDense}, $p$ is onto. Thus (\ref{it:onto}) holds.

By Theorem \ref{thm:HCAdm}, $E_{HC}^{d}(HC(\pi))$ is admissible. Thus  $HC(E_{\infty}^{d}(\pi))$ is admissible and thus $E_{\infty}^{d}(\pi) \in \cM_{\infty}(G_{n-d})$.
By Corollary \ref{cor:HCNilp}, $\fv_{n-d+1}$ acts nilpotently on $E_{HC}^{d}(HC(\pi))$, and hence, by (\ref{it:onto}), on $HC(E_{\infty}^{d}(\pi))$ and hence, by continuity, $\fv_{n-d+1}$ acts nilpotently on
$E_{\infty}^{d}(\pi)$. This implies (\ref{it:1=2}) and (\ref{it:d+1=0}).
\end{proof}

Note that this does not prove that the $d$-th derivative is non-zero. However it is true; see the remarks following Theorem \ref{thm:main} in the introduction.


\begin{cor}
Let $\pi \in \cM_{\infty}(G_n)$ be of depth $d$. Then
$B_{\infty}^{d}(\pi) \in \cM_{\infty}(G_{n-d})$.
\end{cor}
\begin{proof}
By Corollary \ref{cor:SmoothDer}, $E_{\infty}^{d}(\pi) \in \cM_{\infty}(G_{n-d})$. Note that $B_{\infty}^{d}(\pi)$ is the cokernel of the action map $a:\fv_{n-d+1} \otimes E_{\infty}^{d}(\pi) \to E_{\infty}^{d}(\pi)$. Thus, it is enough to show that $\Im(a)$ is closed.
This follows from Corollary \ref{cor:CW}.
\end{proof}

\Dima{
\begin{lem}\label{AnnVarBE}
For any $\pi \in \cM_{HC}^d(G_n)$,
$\cV(B^{d}(\pi))=\cV(E^d(\pi))$.
\end{lem}
\begin{proof}
By Corollary \ref{cor:HCNilp}, there exists $k$ such that $\fv_{n-d+1}^k E^d(\pi) =0$. Consider the descending filtration $F^i(E^d(\pi)):= \fv_{n-d+1}^i E^d(\pi)$. Then $Gr^0(E^d(\pi))=B^d(\pi)$, $Gr^i(\pi) \in \cM_{HC}(G_{n-d})$ and we have a natural morphism $Gr^i(\pi) \otimes \fv _{n-d+1} \onto Gr^{i+1}(\pi)$, where we view $\fv_{n-d+1}$ as the standard representation of $G_{n-d}$. Thus,
$\cV(E^d(\pi)) = \cup_{i=1}^k \cV(Gr^i(E^d(\pi)))$ and $\cV(Gr^i(E^d(\pi)))\supset \cV(Gr^{i+1}(E^d(\pi)))$. Thus $\cV(B^{d}(\pi))=\cV(Gr^0(E^d(\pi))= \cV(E^d(\pi))$.
\end{proof}
By Corollary \ref{cor:SmoothDer} the same statement holds for $\pi \in \cM_{\infty}^d(G_n)$ and the same proof works.
}

\begin{remark}
\cite[Theorem 5.0.4]{GS-Gen} gives a formula for  $\cV(B^{k}(\pi))$ in terms of $\cV(\pi)$ and $k$, for any $k$.
\end{remark}

\begin{thm}[\cite{AGS2}, Theorem B]\label{thm:Prod}
Let $n=n_1+\cdots+n_k$ and let $\chi_i$ be characters of $G_{n_i}$. Let $\pi= \chi_1 \times \cdots \times \chi_k$ denote the corresponding monomial representation. Then 
$$E_{\infty}^k(\pi) \Dima{= E^1_{\infty}(\chi_1) \times \cdots \times E^1_{\infty}(\chi_k) }= ((\chi_1)|_{G_{n_1-1}} \times \cdots \times (\chi_k)|_{G_{n_k-1}}).$$
\end{thm}


\begin{remark}
$ $
\begin{enumerate}
\item Note that $depth(\pi)=k$ by Corollary \ref{cor:DepthProd}.

\item In the special case $n=k$ this theorem implies that the space of Whittaker functionals on a principal series representation is one-dimensional. By the same example we see that an analog of Theorem \ref{thm:Prod}  for $E_{HC}$ does not hold, since the space of Whittaker functionals on the Harish-Chandra module of a principal series representation has dimension $n!$.
\end{enumerate}
\end{remark}

From Theorem \ref{thm:Prod}, Theorem \ref{thm:ExactHaus} and Corollary \ref{cor:SmoothDer} we obtain
\begin{cor}
Let  $I= \chi_1 \times \cdots \times \chi_k$ be a monomial representation. Let $\pi$ be any subquotient of $I$. Then $E_{\infty}^k(\pi)\cong D_{\infty}^k(\pi)\cong B_{\infty}^k(\pi)$.
\end{cor}
\begin{proof}
If $depth(\pi)<k$ we have $E_{\infty}^k(\pi)=D_{\infty}^k(\pi)=B_{\infty}^k(\pi)=0$. Otherwise, $depth(\pi)=k$,  $E_{\infty}^k(\pi)\cong D_{\infty}^k(\pi)$ and we have to show that $\fv_{n-k+1}E_{\infty}^k(\pi)=0$. Note that all the subquotients of a monomial representation have the same infinitesimal character.
By Theorem \ref{thm:Prod}, $E_{\infty}^k(I)$ is again a monomial representation and by Theorem \ref{thm:ExactHaus} $E_{\infty}^k(\pi)$ is its subquotient. Thus, the identity matrix $Id \in \g_{n-k}$ acts on $E_{\infty}^k(\pi)$ by a scalar, that we denote by $a$. Now, let $x \in E_{\infty}^k(\pi)$ and $v \in \fv_{n-k+1}$. Then $avx=Idvx= [Id,v]x + vIdx=vx+avx=(a+1)vx$. Thus, $vx=0$ for any $v$ and $x$ and thus $\fv_{n-k+1}E_{\infty}^k(\pi)=0$ as required.
\end{proof}

In \RamiA{\S \ref{sec:App}}
we demonstrate several applications of Theorem \ref{thm:Prod} to unitary representations.

\Dima{
\subsection{Conjectures and questions}

We conjecture that the following generalization of Theorem \ref{thm:Prod} holds for Bernstein-Zelevinsky product of arbitrary representations.

\begin{conj}\label{conj:Prod}
Let $n=n_1+n_2$ and let $\pi_i \in \cM_{\infty}^{d_i}(G_{n_i})$ for $i=1,2$. Let $\pi:= \pi_1 \times \pi_2$ and $d:=d_1+d_2$, so that $\pi\in \cM_{\infty}^d(G_n)$.
Then 
\RamiA{
$E_{\infty}^{d}(\pi)=E_{\infty}^{d_1}(\pi_1)  \times E_{\infty}^{d_2}(\pi_2)$.
}
\end{conj}

We think that it is possible to prove this conjecture using the same geometric argument as in the proof of Theorem \ref{thm:Prod} (in \cite{AGS2}).
We believe that the main ingredient we miss are suitable notions of tempered \Fre bundle and the space of its Schwartz sections.

Our paper leaves the following open questions:
\begin{enumerate}
\item In Proposition \ref{prop:D3Adm} we show that $B^k$ maps admissible Harish-Chandra modules to admissible, for every $k$. Is the same true on the smooth category $\cM_{\infty}(G_n)$? In other words, is $B^k(\pi)$ Hausdorff for any $\pi \in \cM_{\infty}(G_n)$ and any $k$? We think that the answer is yes and hopefully it can be proven using the methods of  \cite{AGS2}.
\item Is the depth $B$-derivative functor $B^d: \cM^d_{\infty}(G_n) \to \cM_{\infty}(G_{n-d})$ exact?
\item Is $D^k$ exact for any $k$?
\item Does $B^d$ map irreducible representations of depth $d$ to irreducible ones? This would in particular imply uniqueness of degenerate Whittaker models for all smooth representations (see \S \ref{sec:App}).
\end{enumerate}
}

\section{\Dima{Highest derivatives  of unitary representations and applications}}\label{sec:App}
In this section we prove two results conjectured in \cite{GS}: uniqueness of degenerate Whittaker models and computation of adduced representations for Speh complementary  series. We also prove  that for irreducible unitary representations, the three notions of highest derivative
coincide
and extend the notion of adduced representation.
We start with preliminaries on Speh representation, Vogan classification, adduced representation, and degenerate Whittaker models.

\subsection{Preliminaries}\label{subsec:UniPrel}$ $

\DimaA{
\begin{notn}
Let $z \in \C,$ and let $\eps \in \Z/2\Z$ if $F=\R$ and $\eps \in \Z$ if $F=\C$.
Denote by $\chi\left(  n,\varepsilon,z\right)  $ the character of $G_n$ given by
\[
x\mapsto\left( \frac{\det x}{|\det x|} \right)  ^{\varepsilon}\left\vert \det
x\right\vert ^{z}.
\]
This
character is unitary if $z$ is imaginary.
\end{notn}
}

\subsubsection{Speh representation}

We will use the following description of the Speh representation form \cite{BSS,SaSt}.

\begin{prop}\label{prop:SpehDegPrin}
The Speh representation $\delta\left(  2m,k\right)  $ of $GL_{2m}(\R)$ is a
quotient of
\begin{equation*}
\chi\left(  m,\varepsilon_{k+1},-k/2\right)  \times\chi\left(  m,0,k/2\right) \end{equation*} and a
submodule of
\begin{equation*}
\chi\left(  m,\varepsilon_{k+1},k/2\right)  \times\chi\left(  m,0,-k/2\right)
\end{equation*}
where $k\in\mathbb{N}$ and $\varepsilon_{k+1}\equiv k+1\left(  \operatorname{mod}2\right).$
\end{prop}

\DimaA{
\begin{rem}
By Lemma \ref{lem:UnSubq} and Proposition \ref{prop:MonIrr} below, the Speh representation $\delta\left(  2m,k\right)  $ is
the unique irreducible
submodule of
$
\chi\left(  m,\varepsilon_{k+1},k/2\right)  \times\chi\left(  m,0,-k/2\right)
$
and thus also the unique irreducible
quotient of
$
\chi\left(  m,\varepsilon_{k+1},-k/2\right)  \times\chi\left(  m,0,k/2\right) $.
\end{rem}
}
\subsubsection{Vogan Classification} \label{subsubsec:VogClas}

By the Vogan classification \cite{Vog-class}, irreducible unitary
representations of $G_n$ are BZ
products of the form
\[
\pi=\pi_{1}\times\cdots\times\pi_{k}%
\]
where each $\pi_{i}$ is one of the following basic unitary representations:

\begin{enumerate}[(a)]
\item \label{it:Char}\emph{A one-dimensional unitary character of some }$G_m$. 

\item \label{it:StCS}\emph{A Stein complementary series representation of some }$G_{2m}$, twisted by a unitary character. The Stein
representations are complementary series of the form%
\[
\sigma\left(  2m,s\right)  =\chi\left(  m,0,s\right)  \times\chi\left(
m,0,-s\right), s\in\left(  0,1/2\right)
\]
and we write $\sigma\left(  2m,s;\varepsilon,it\right)  $ to denote its twist
by$\chi\left(  2m,\varepsilon,it\right)  $.

\item \label{it:Speh} \emph{A Speh representation} $\delta(2m,k)$ of  $G_{2m}$,
twisted by a unitary character.
We write $\delta\left(  2m,k;it\right)  $ to denote its twist by
$\chi\left(  2m,0,it\right)  $.

\item  \label{it:SpehCS} \emph{A Speh complementary series representation of some }$G_{4m}$, twisted by a unitary character. The Speh
complementary series representation is%
\[
\Delta\left(  4m,k,s\right)  =\delta\left(  2m,k;0,s\right)  \times \delta \left(
2m,k;0,-s\right), s\in\left(  0,1/2\right)
\]
and we write $\Delta\left(  4m,k,s;it\right)  $ to denote its twist by
$\chi\left(  4m,0,it\right)  $.
\end{enumerate}

It is known that the associated partition of $\delta(  2m,k)$ is $2^m$ and (thus) the associated partition of $\Delta\left(  4m,k,s\right)$ is $4^m$ (see e.g. \cite[Theorem 4.2.1]{GS}).

If $F=\C$ then only types (\ref{it:Char}) and (\ref{it:StCS}) appear in the Vogan classification. Thus, every $\tau \in \widehat{GL(n,\C)}$ is a product of (not necessary unitary) characters. For $F=\R$ this is not true, but the above information implies the following result.
\begin{cor}[\cite{GS}, Corollary 4.2.5]\label{cor:UniDegPrin}
Let $\tau \in \widehat{G_n}$, let $\lambda$ be the associated partition of $\tau$ and $\mu = (n_1 , \dots ,  n_k)$ be the transposed partition. Then there exist (not necessarily unitary) characters $\chi_i,\chi'_i$ of $G_{n_i}$ for all $1 \leq i \leq k$, an epimorphism  $\chi_1 \times \cdots \times \chi_k \onto \tau^{\infty}$ and an embedding $\tau^{\infty} \hookrightarrow \chi'_1 \times \cdots \times \chi'_k $. Moreover, $\cV(\chi_1 \times \cdots \times \chi_k) = \cV(\chi'_1 \times \cdots \times \chi'_k) = \cV(\tau^{\infty})$.
\end{cor}

\subsubsection{Adduced representation}\label{subsubsec:PrelA}

Adduced representation was defined in \cite{Sahi-Kirillov} for irreducible unitary representations. The definition is based on the fact that an irreducible unitary representation of $G_n$ remains irreducible (as a unitary representation) after restriction to $P_n$ (see \cite{Bar}). The adduced representation of $\tau \in \widehat{G_n}$, denoted $A\tau$
is defined to be the irreducible unitary representation of $G_{n-d}$ that gives rise to $\tau|_{P_n}$ by Mackey induction. The number $d$ is called the depth of $\tau$. By Theorem \ref{thm:GSMain} below, $depth(\tau)=depth(\tau^{\infty})$.

Clearly, $A\chi = \chi|_{G_{n-1}}$ for any unitary character $\chi$. It is also clear that $A$ ``commutes" with a twist by a unitary character.
In \cite{Sahi-Kirillov}, it was proven that $A$ is a multiplicative operation (an analog of Conjecture \ref{conj:Prod}).
In \cite{Sahi-PAMS}, it was shown that $A\sigma(2m,s) = \sigma(2m-2,s)$, which completed the computation of adduced representations for $F=\C$. In \cite{SaSt} it was shown that $A\delta(2m,k)=\delta(2m-2,k)$. Finally, in \cite{GS} it was shown, using the Vogan classification, that $A\Delta(4m,k,s) = \Delta(4m-4,k,s)$ if $k \neq m$ and conjectured that the condition $k \neq m$ is not necessary.
Of course, this conjecture immediately follows from Conjecture \ref{conj:Prod} and the fact that $A\delta(2m,k)=\delta(2m-2,k)$. However, in this section we prove that $A\Delta(4m,k,s) = \Delta(4m-4,k,s)$ for all $k$ and $m$ in a different way (without using the Vogan classification). This completes the computation of adduced representations for all irreducible unitary  representations of $G_n$.

\Dima{We also prove that $\RamiA{(A\tau)^{\infty}} \cong B^d(\tau^{\infty})\cong D^d(\tau^{\infty})\cong E^d(\tau^{\infty})$.}

The following lemma follows from the definition of adduced representation and Frobenius reciprocity. We refer the reader to Notation \ref{not:main} for the definitions of the groups $S_{n}^d$ and $U_n^d$.
\begin{lem}[\cite{GS}, Proposition 3.1.2]
\label{lem:pi2Api}  Let $\tau\in\widehat{G_{n}}$, and let $d:=depth(\tau)$.
Extend the action of $G_{n-d}$ on $(A\tau)^{\infty}$ to an action of $S_{n}^d$ by letting $U_n^d$ act by the character $\psi^d$.

Then
there exists an $S_{n}^d$-equivariant map from $\tau^{\infty}$ to
$(A\tau)^{\infty}\otimes|\det|^{(d-1)/2}$ with dense image.
\end{lem}

\begin{cor}\label{cor:D3OntoA}
Let $\tau\in\widehat{G_{n}}$, and let $d:=depth(\tau)$. Then there is a natural epimorphism $B_{\infty}^{d}(\tau^{\infty})\onto (A\tau)^{\infty}$.
\end{cor}
\begin{proof}
Twisting the map from the previous lemma by $|\det|^{-(d-1)/2}$ we obtain a map $\tau^{\infty} \otimes |\det|^{-(d-1)/2} \to (A\tau)^{\infty}$. Since it is $S_{n}^d$-equivariant, it factors through $B_{\infty}^{d}(\tau^{\infty})$. Thus we have a map $B_{\infty}^{d}(\tau^{\infty})\to (A\tau)^{\infty}$ with dense image. By the Casselman-Wallach theorem (see Corollary \ref{cor:CW}) the image is closed and hence this is an epimorphism.
\end{proof}

\subsubsection{Degenerate Whittaker models} \label{subsubsec:DegWhit}

For a composition $\lambda = (n_1 , \dots ,  n_k)$ of $n$, let $J_{\lambda}$ denote the corresponding upper triangular matrix consisting of Jordan blocks of sizes $n_1 , \dots ,  n_k$. Let $w_0J_\lambda w_0^{-1}$ be the conjugation of  $J_{\lambda}$ by the longest Weyl group element $w_0$. Let $\psi_{\lambda}$ denote the character of $\fn$ given by $$\psi_{\lambda}(X):=\sqrt{-1}\pi \re (\tr(Xw_0J_{\lambda}w_0^{-1}))=d\theta(\tr(Xw_0J_{\lambda}w_0^{-1}))$$
By abuse of notation, we denote the corresponding character of $N$ also by $\psi_{\lambda}$.

\begin{defn} Let $\pi \in \cM_{\infty}(G_n)$ and $\lambda$ be a composition of $n$.
\begin{itemize}
\item  Denote $Wh_{\lambda}^*(\pi):=\Hom_{N}(\pi, \psi_{\lambda})$ and for $\tau \in \widehat{G_n}$ denote $Wh_{\lambda}^*(\tau):=Wh_{\lambda}^*(\tau^{\infty})$.

\item Denote $E^{\lambda}(\pi):=E^{n_k}(\cdots(E^{n_1}(\pi)|_{\p_{n-n_1}})\cdots)$ and  $B^{\lambda}(\pi):=B^{n_k}(\cdots(B^{n_1}(\pi)|_{\p_{n-n_1}})\cdots)$.

\end{itemize}
\end{defn}

The following lemma is obvious.

\begin{lem}$ $
\begin{enumerate}
\item $Wh^*_{\lambda}(\pi) \cong (B^\lambda(\pi))^*$
\item We have a natural epimorphism $E^\lambda(\pi) \onto B^\lambda(\pi)$.
\end{enumerate}
\end{lem}


We will use the main result of \cite{GS}.

\begin{thm}[\cite{GS}, Theorem A]\label{thm:GSMain}
Let $\tau\in\widehat{G_{n}}$ and let $\lambda=(n_1 , \dots ,  n_k)$ be the associated partition of $\tau$.
Then
\begin{enumerate}

\item $Wh_{\lambda
}^{\ast}\left(  \tau\right)   \neq 0$.

\item $depth(\tau) = n_1$ and the associated partition of $A\tau$ is $(n_2 , \dots ,  n_k)$.
\end{enumerate}
\end{thm}
We will show that $\dim Wh_{\lambda
}^{\ast}\left(  \tau\right)=1$ (see Theorem \ref{thm:GSMult1}).

\subsection{Applications}\label{subsec:UniApp}

We will use the following immediate corollary of Theorem \ref{thm:Prod}.

\begin{cor}\label{cor:DlamProd}
Let $\alp=(n_1 , \dots ,  n_k)$ be a composition of $n$, $\lam$ be the partition obtained by reordering of $\alp$ and $\mu$ be the transposed partition of $\lam$. Let $\chi_i$ be characters of $G_{n_i}$ for $1 \leq i \leq k$. Then $$\dim E^{\mu}(\chi_1 \times  \cdots \times \chi_k) = 1$$
\end{cor}

By exactness of $E$ (Theorem \ref{thm:ExactHaus}) we obtain
\begin{cor}\label{cor:UnSubqDegPS}
In the notations of the previous corollary, $\chi_1 \times  \cdots \times \chi_k$ has a unique irreducible subquotient $\pi$ such that $\dim E^{\mu}(\pi) =1$. For any other irreducible subquotient $\rho$ we have $E^{\mu}(\rho)=B^{\mu}(\rho)=0$.
\end{cor}

First of all, let us prove the following strengthening of Theorem \ref{thm:GSMain}.

\begin{thm}\label{thm:GSMult1}
Let $\tau\in\widehat{G_{n}}$ and let $\lambda=(n_1 , \dots ,  n_k)$ be the associated partition of $\tau$.
Then
$\dim Wh_{\lambda
}^{\ast}\left(  \tau\right)=1$.
\end{thm}
\begin{proof}
Let $\mu=(m_1 , \dots ,  m_l)$ be the transposed partition to $\lambda$.
By Corollary \ref{cor:UniDegPrin},
there exist characters $\chi_i$ of $G_{m_i}$ for all $1 \leq i \leq l$ and an epimorphism  $I_{\mu}:=\chi_1 \times  \cdots  \times \chi_l \onto \tau^{\infty}$.
By Corollary \ref{cor:DlamProd}, $\dim E^{\lam}(I_{\mu})=1$ and hence $$\dim Wh^*_{\lambda}(\tau) \leq  \dim B^\lam(\tau^{\infty}) \leq
\dim B^\lam(I_{\mu}) \leq \dim E^\lam(I_{\mu}) =1.$$
Since by Theorem \ref{thm:GSMain} $Wh^*_{\lambda}(\tau) \neq 0,$ we have $\dim Wh_{\lambda}^{\ast}\left(  \tau\right)=1$.
\end{proof}

Let us now compute the adduced representations of the Speh complementary series representations.
\begin{thm}\label{thm:ASCS}
$A(\Delta\left(  4m,k,s\right)) \cong \Delta\left(  4m-4,k,s\right).$
\end{thm}
\begin{proof}
Since the associated partition of $\Delta\left(  4m,k,s\right)$ is $4^m$, Theorem \ref{thm:GSMain} implies that $depth(\Delta\left(  4m,k,s\right))=4$ and the associated partition of $A(\Delta\left(  4m,k,s\right)$ is $\RamiA{(4)}^{m-1}$ which is the same as that of $\Delta\left(  4m-4,k,s\right).$
Let $$I_{m,k,s}:= \chi\left(  m,\varepsilon_{k+1},-k/2 + s\right)  \times\chi\left(  m,0,k/2 + s \right) \times \chi\left(  m,\varepsilon_{k+1},-k/2 - s\right)  \times\chi\left(  m,0,k/2 - s \right).$$
From \S \ref{subsubsec:VogClas} we see that $\Delta\left(  4m,k,s\right)$ is a quotient of $I_{m,k,s}$.
Thus, 
$E^4_{\infty}(\Delta\left(  4m,k,s\right))$ is a quotient of $E^4_{\infty}(I_{m,k,s})$. By Theorem \ref{thm:Prod}, $E^4_{\infty}(I_{m,k,s})=I_{m-1,k,s}$ and by Corollary \ref{cor:D3OntoA} $A(\Delta\left(  4m,k,s\right))$ is a quotient of $E^4_{\infty}(\Delta\left(  4m,k,s\right))\otimes|\det|^{1/2}$. Altogether, we get that $A\Delta\left(  4m,k,s\right)$ is a quotient of  $I_{m-1,k,s}$. But so is $\Delta\left(  4m-4,k,s\right)$.

From Theorem \ref{thm:GSMain} we have $Wh^*_{4^{m-1}}(A(\Delta\left(  4m,k,s\right)) \neq 0$ and $Wh_{4^{m-1}}^*(\Delta\left(  4m-4,k,s\right)) \neq 0$. Thus, Corollary \ref{cor:UnSubqDegPS} implies $A(\Delta\left(  4m,k,s\right)) \cong \Delta\left(  4m-4,k,s\right)$.
\end{proof}

\begin{remark}
Using a similar argument, one can give an alternative proof of $A\delta(2m,k) \cong \delta(2m-2,k)$, using only the fact that $\delta(2m,k)$ is an irreducible quotient of  $\chi\left(  m,\varepsilon_{k+1},-k/2\right)  \times\chi\left(  m,0,k/2\right)$ with associated partition $2^m$.
\end{remark}

\subsection{The isomorphisms between depth derivatives and adduced representation}\label{subsec:HDA}


In this subsection we prove that for irreducible unitary representations, all the notions of highest derivative agree and extend the notion of adduced representation.

\begin{thm}\label{thm:HDA}
For $\tau \in \widehat{G_n}$ of depth $d$, the canonical maps
$$E_{\infty}^{d}(\tau^{\infty}) \to D_{\infty}^{d}(\tau^{\infty}) \to B_{\infty}^{d}(\tau^{\infty})\to (A\tau)^{\infty}$$
 are isomorphisms.
\end{thm}

For case $F=\C$ it follows from the Vogan classification ( \S \ref{subsubsec:VogClas}), Theorem \ref{thm:Prod} and Corollary \ref{cor:D3OntoA}.
Thus from now on till the end of the section, we let $F=\R$. It is enough to prove that the resulting map $E_{\infty}^{d}(\tau^{\infty}) \onto (A\tau)^{\infty}$  has zero kernel.

\begin{defn}
To every $\tau \in \widehat{G_n}$ of depth $d$ we attach a monomial representation $I^\geq(\tau)$.
First, we attach to $\tau$ $d$ characters using the Vogan classification in the following way.
To every character we attach itself, to the Speh representation $\delta(m,k)$ we attach $\chi\left(  m,\varepsilon_{k+1},k/2\right)$ and $\chi\left(  m,0,-k/2\right)$ and to the Speh complementary series $\Delta(m,k,s)$ we attach $\chi\left(  m,\varepsilon_{k+1},k/2+s \right)$, $\chi\left(  m,0,-k/2 +s \right)$, $\chi\left(  m,\varepsilon_{k+1},k/2 - s \right)$ and $\chi\left(  m,0,-k/2 -s\right)$. Then we order all the characters we have in non-ascending order of the real part of the complex parameter, and let $I^\geq(\tau)$ be their BZ product in this order.
\end{defn}

Note that by Proposition \ref{prop:SpehDegPrin} and Corollary \ref{cor:UniDegPrin}, $\tau^{\infty}$ is a subquotient of $I^\geq(\tau)$ and $\cV(\tau)=\cV(I^\geq(\tau))$.

\begin{lem}\label{lem:UnSubq}
For every $\tau \in \widehat{G_n}$, the smooth representation $\tau^{\infty}$ occurs with multiplicity one in the Jordan-Holder series of $I^\geq(\tau)$ and it is the only \RamiA{irreducible} subquotient of $I^\geq(\tau)$ with maximal annihilator variety.
\end{lem}
\begin{proof}

Let $\delta$ denote a Speh representation.
By Proposition \ref{prop:SpehDegPrin}, $\delta \subset I^{\geq}(\delta)$ and by \cite[Proposition VI.7]{BSS}, $\cV(I^{\geq}(\delta)/\delta)\subsetneq \cV(I^{\geq}(\delta))$.
Now let $\tau \in \widehat{G_n}$. Then $\tau=\prod \chi_i \times \prod \delta_j$, where $\chi_i$ are characters, and $\delta_j$ are Speh representations, possibly twisted by characters. Then, by exactness of tensor product,
$(\boxtimes_i \chi_i) \boxtimes (\boxtimes_j I^{\geq}(\delta_j)) \supset  (\boxtimes_i \chi_i) \boxtimes (\boxtimes_j \delta_j)$ and the annihilator variety of the quotient is strictly smaller then the annihilator variety of $(\boxtimes_i \chi_i) \boxtimes (\boxtimes_j I^{\geq}\delta_j)$. The lemma now follows from the exactness of induction, Theorem \ref{thm:AnnVarProd} and the fact that $I^{\geq}(\tau)$ and $\prod \chi_i \times \prod I^{\geq}(\delta_j)$ have the same Jordan-Holder components.

\end{proof}

\begin{rem}
One can show that any monomial representation has a unique subquotient with maximal annihilator variety, occuring with multiplicity one.
By \cite{GS-Gen}, this is equivalent to the statement that it has a unique subquotient with minimally degenerate Whittaker functional, which follows from  Corollary \ref{cor:UnSubqDegPS}.
\end{rem}

\begin{lem}\label{lem:IA}
For $\tau \in \widehat{G_n}$ of depth $d$, $I^{\geq}(A\tau)=E^d(I^{\geq}(\tau))$.
\end{lem}
This lemma follows from Theorem \ref{thm:Prod}, 
\S\ref{subsubsec:PrelA} and Theorem \ref{thm:ASCS}.

We will also use the following proposition that we will prove in \S \ref{subsubsec:PfMonIrr}.

\begin{prop}\label{prop:MonIrr}
Let $\chi_1 , \dots ,  \chi_k$ be characters of $G_{n_1} , \dots ,  G_{n_k}$ respectively and let $d\chi_i$ denote the corresponding characters of $\g_{n_i}$. Suppose that $\re d\chi_1 \geq  \cdots  \geq \re d\chi_k$.
Then any subrepresentation $\pi \subset \chi_1 \times  \cdots  \times \chi_k$ satisfies $\cV(\pi) = \cV(\chi_1 \times  \cdots  \times \chi_k)$.

\end{prop}

\begin{proof}[Proof of Theorem \ref{thm:HDA}]
Let $\rho$ be the kernel of the canonical map $E^d(\tau^{\infty}) \onto A\tau^{\infty}$ (given by Corollary \ref{cor:D3OntoA}). We have to show that $\rho$ is zero.

By Lemma \ref{lem:UnSubq} $\tau^{\infty}$ is the only irreducible subquotient of  $I^{\geq}(\tau)$
with maximal annihilator variety. By Proposition \ref{prop:MonIrr} such a subquotient is necessarily a submodule. Thus  we have an embedding $\tau^{\infty} \subset I^{\geq}(\tau)$. By Theorem \ref{thm:ExactHaus} and Lemma \ref{lem:IA} we have $E^d(\tau^{\infty}) \subset E^d(I^{\geq}(\tau))=I^{\geq}(A\tau)$. Thus, we get that $\rho$ is a subrepresentation of $I^{\geq}(A\tau)$ and thus, by Proposition \ref{prop:MonIrr}, either $\rho=0$ or $\cV(\rho) = \cV(I^{\geq}(A\tau)) = \cV(A\tau)$. The second option contradicts Lemma \ref{lem:UnSubq} and thus $\rho=0$.
\end{proof}

\subsubsection{Proof of Proposition \ref{prop:MonIrr}}\label{subsubsec:PfMonIrr}
We will use the following lemmas.

\begin{lem}\label{lem:VermaIrr}
Let $\fq$ be a standard parabolic subalgebra of $\g_n$ with nilradical $\fu$ and Levi subalgebra $\fl$. Let $\lambda=(\lambda_1 , \dots ,  \lambda_k)$ be a character of $\fl$. Let $M_{\lambda}:=\U(\g) \otimes _{\U(\fq)}(\lambda -\rho(\fu))$ be the corresponding (normalized) generalized Verma module. Suppose that $\lambda_1\geq  \cdots  \geq \lambda_k$. Then $M_{\lambda}$ is irreducible.
\end{lem}
This follows e.g. from \cite[Lemma 3.5]{Trapa} or \cite[Satz 4 and Corollar 4]{Jan}.

\begin{lemma}
\label{lem:IndPairing}[see e.g. \cite{GS}, Lemma 2.3.9] Let $G$ and $Q$ be Lie groups, $\sigma$ be a smooth representation of $Q$ and $\left(  \pi,W\right)  =Ind_{Q}^{G}(\sigma)$ be the (normalized) induction of $\sigma$.
Suppose that $G/Q$ is compact and connected and $(\sigma \otimes  \RamiA{\Delta} _Q^{1/2} \otimes  \RamiA{\Delta} _G^{-1/2})^{\omega}$ is
non-degenerately ${\mathfrak{q}}$-invariantly paired with a ${\mathfrak{q}}$-module
$X$. Then $\pi^{\omega}$ is non-degenerately ${\mathfrak{g}}$-invariantly paired with a
quotient of $U({\mathfrak{g}})\otimes_{U({\mathfrak{q}})}X$, where $\omega$ denotes the space of analytic vectors.
\end{lemma}

\begin{proof}[Proof of Proposition \ref{prop:MonIrr}]
Let $I= \chi_1 \times  \cdots  \times \chi_k$, $Q:=P_{(n_1 , \dots ,  n_k)}$ and $\lambda:=(\chi_1 , \dots ,  \chi_k)$.
Let $M_{\lambda}:=\U(\g) \otimes _{\U(\fq)}(\lambda -\rho(\fu))$ be the corresponding (normalized) generalized Verma module. By Lemma \ref{lem:VermaIrr} and the assumptions of the proposition, $M_{\lam}$ is irreducible.
Thus, by Lemma \ref{lem:IndPairing}, $M_\lam$ is non-degenerately paired with $I^{\omega}$ and hence also with $\pi^{\omega}$ for any submodule $\pi \subset I$.

Note that annihilators of two non-degenerately paired representations are equal and that
$Ann(\pi)=Ann(\pi^{\omega})= Ann(\pi^{HC})$, since $\pi^{HC}$ includes $\pi^{\omega}$ and is dense in $\pi$.
Thus, the annihilators of $\pi$, $M_\lam$ and $I$ are equal and thus $\cV(\pi)=\cV(M_{\lam})=\cV(I)$.
\end{proof}

\section{Proof of admissibility (Theorem \ref{thm:HCAdm})} \label{sec:Adm}
\setcounter{lemma}{0}

\subsection{Structure of the proof}\label{subsec:AdmPfStr}

First of all, by Theorem \ref{thm:EqAdm}, a $(\g,K)$-module $\pi$ is admissible if and only if it is finitely generated over $\n$.
Thus, we know that $E^d(\pi)$ is finitely generated over $\n_{n-d+1}$ and we need to show that it is in fact finitely generated over $\n_{n-d}$.

Let us consider an analogous situation in commutative algebra.
Let $W$ be a finite-dimensional complex vector space, and 
$W =U \oplus V \oplus W'$ be a direct sum decomposition.
Let $\xi \in U^*$.
Note that it defines
a homomorphism of algebras 
$\DimaA{\Xi}:\Sym(W) \to \Sym(V \oplus W')$. The corresponding map $(V \oplus W')^* = \Spec(\Sym(V \oplus W')) \to \Spec(\Sym(W))=W^*$ is given just by adding $\xi$.
Let $I:=\Ker\DimaA{\Xi}$.
Let $M$ be a finitely-generated module over $\Sym(W)$ and $L:=M/I$. Suppose that we know the support of $M$ and want to find out when $L$ is finitely generated over $\Sym(W')$. A sufficient condition will be   $\Supp M \subset \xi + (W')^* + 0$. Indeed, this condition is in turn equivalent to
$pr_{V}(\Supp L)=0$, which implies that $V$ acts on $L$ nilpotently and thus is not needed for finite generation.


Our proof follows the general lines of this argument. However, over non-commutative Lie algebras we have several problems. The first is that we do not have a straightforward notion of support. One substitute is the annihilator variety. As we recall in \S \ref{subsubsec:AnnVar}, it is defined by considering the annihilator ideal of the module in $U(\g)$, then passing to the graded ideal in $\Sym(\g)$ and considering its zero set.
In our proof we also use a second notion of support, called associated variety, that we recall in \S \ref{subsec:Filt}. It is a finer invariant, defined by first passing to associated graded module and then considering the support in $\Spec\g^*=\Sym(\g)$. In order to define this invariant, one has to fix a filtration on the module $M$. If $M$ is finitely generated, it has so-called ``good" filtrations (see \S \ref{subsubsec:AnnVar}) and the associated variety $AV(M)$ does not depend on the choice of a good filtration.

For finitely-generated modules over commutative Lie algebras, $\cV(M) = AV(M)$. However, they might be different from $\Supp(M)$, because $AV$ and $\cV$ are always conical sets. For example, let $\g = \bC$,  $U(\g)=\Sym(\g)=\C[x]$  and $M:=\C[x]/(x-1)$. Then $\Supp(M) = \{1\}$ and $AV(M)=\cV(M)=0$. Moreover, over non-commutative algebras the associated variety of a quotient is not determined by the associated variety of $M$.

Thus, we will have to use the fact that our module comes from $\g_n$, and thus its annihilator ideal is invariant under the torus action. Thus, it is a graded ideal with respect to the grading defined by any torus element. In order to connect the associated variety with respect to actions of $\g_n$ and $\fs_n$ we will use a proposition saying that any  good $\g_n$-filtration on $\pi$ is also a good $\n_n$-filtration (see Proposition \ref{prop:WonFilt}).

Another difficulty we have in the non-commutative case is that the quotient of $\pi \otimes \psi^d$ by $\fu^d$ is not a module over $ \fs_d/\fu^d = \g_{n-d+1},$ but only over $\p_{n-d+1}$. In order to cope with this difficulty we use a second torus element, and thus consider a bigrading on $U(\g_n)$.

We introduce the required notations and terminology on bigradings in \S \ref{bigraded} and on the module of coinvariants in \S \ref{coinv}.
In \S \ref{subsec:keylem} we prove the key lemma that connects the associated variety of $E^d(\pi)$ to the annihilator variety of $\pi$.
This lemma is a generalization of \cite[Theorem 1]{Mat}.
Finally, in \S \ref{subsec:PfHCAdm} we deduce that $pr_{V_{n-d+1}^*}(AV(E^d(\pi))) = \{0\}$ and finish the proof of the theorem.

%

\subsection{Filtrations and associated varieties}
\label{subsec:Filt}

In this subsection we define the associated variety of a (finitely generated) module over a Lie algebra. For that we will need the notion of filtration.
Throughout the subsection, let $A$ denote a not
necessarily commutative algebra with $1$ over $\mathbb{C}$.

\begin{defn}
A filtration on $A$ is a presentation of $A$ as an ascending union of vector
spaces $A=\bigcup_{i\geq0} F^{i}A$ such that $F^{i}A F^{j}A \subset F^{i+j}A$.
A filtration on an $A$-module $M$ is a presentation of $M$ as an ascending
union of vector spaces $M=\bigcup_{i\geq0} F^{i}M$ such that $F^{i}A F^{j}M
\subset F^{i+j}M$.

If $A$ is a filtered algebra we define the associated graded algebra
$\operatorname{Gr}(A):=\bigoplus_{i\geq0} F^{i+1}A/F^{i}A$ with the obvious
algebra structure. If $M$ is a filtered module we define the associated graded
module $\operatorname{Gr}(M)$ over $\operatorname{Gr}(A)$ by
$\operatorname{Gr}(M):=\bigoplus_{i\geq0} F^{i+1}M/F^{i}M$.

\end{defn}

\begin{defn}
Let $A$ be a filtered algebra and $M$ be an $A$-module. A filtration $F^{i}M$
is called good if

\begin{enumerate}
[(i)]

\item Every $F^{i}M$ is a finitely generated module over $F^{0}A$ and

\item There exists $n$ such that for every $i>n$, $F^{i+1}M=F^{1}AF^{i}M$.
\end{enumerate}

A filtration on $A$ is called good if it is a good filtration of $A$ as a
module over itself.
\end{defn}

\begin{exm}
Let ${\mathfrak{g}}$ be a (finite dimensional) Lie algebra and
$U({\mathfrak{g}})$ the universal enveloping algebra. Define a good filtration
$U^{i}$ on $U({\mathfrak{g}})$ by the order of the tensor. Then
$\operatorname{Gr}(U({\mathfrak{g}})) = \Sym({\mathfrak{g}})$, the symmetric
algebra.

\end{exm}

From now on we fix a good filtration on $A$. Let $M$ be an $A$-module.

\begin{exm}
Suppose that $M$ is finitely generated. Let $\{m_{1} , \dots ,  m_{k}\}$ be a set of
generators. We define a good filtration on $M$ \ by $F^{i}M=\{\sum_{j=1}%
^{k}a_{j}m_{j}$ s.t. $a_{j}\in F^{i}A\}$.
\end{exm}

The following lemma is evident.
\begin{lem}
$ $
\begin{enumerate}[(i)]
\item There exists a good filtration on $M$ if and only if $M$ is finitely
generated over $A$.

\item A filtration $F^{i}M$ is good if and only if $\operatorname{Gr}_{F}(M)$
is finitely generated over $\operatorname{Gr}(A)$.
\end{enumerate}
\end{lem}

\begin{cor}
Suppose that $\operatorname{Gr}(A)$ is Noetherian. Suppose that $M$ is
finitely generated and let $F^{i}M$ be a good filtration on $M$.
Then\newline(i) For any submodule $L \subset M$, the induced filtration
$F^{i}L:=F^{i}M \cap L$ is good.\newline(ii) $A$ is Noetherian.
\end{cor}

In particular, $U({\mathfrak{g}})$ is Noetherian.

An important tool in the proof of Theorem \ref{thm:HCAdm} will be the following proposition.

\begin{prop} \label{prop:WonFilt}
Let $\pi \in \cM_{HC}(G)$ be a Harish-Chandra module and let $F^i$ be a good $\g$-filtration on it. Then $F^i$ is good as an $\n$-filtration, i.e. $F^{i+1} = \n F^i$ for $i$ big enough.
\end{prop}
\RamiA{
This proposition is due to Gabber
and is based on a proposition by Casselman and Osborne.
For completeness we included its proof in Appendix \ref{sec:PfWonFilt}.
}

\begin{defn}
For any filtration $F$ on $M$ we we can associate to $M$ a
subvariety of $\operatorname{Spec} Gr(A)$ by $AV_F(M):=\mathrm{Supp}(Gr(M))
\subset\operatorname{Spec} Gr(A)$.

If $M$ is finitely generated we choose a good filtration $F$ on $M$
and define the associated variety of $M$ to be $AV(M):=AV_F(M)$. This variety does not depend on the choice of the good filtration.
\end{defn}

\begin{remark*}
If $A$ is commutative then  the associated variety equals to the annihilator variety. Otherwise, the associated variety can be smaller.
\end{remark*}

\subsection{Bigraded Lie algebra} \label{bigraded}


Let $\mathfrak{a}$ be a Lie algebra and let $X,Y\in $ $\mathfrak{a}$ be
commuting ad-semisimple elements with integer eigenvalues. Define
\begin{equation*}
\mathfrak{a}_{ij}:=\left\{ Z\in \mathfrak{a}:\left[ X,Z\right] =iZ,\left[ Y,Z%
\right] =jZ\right\}
\end{equation*}%
Thus we have the direct sum decomposition%
\begin{equation}
\mathfrak{a}=\oplus_{i}\mathfrak{a}_{i}\text{ where }\mathfrak{a}_{i}=%
\mathfrak{\oplus }_{j}\mathfrak{a}_{ij}  \label{=dec}
\end{equation}

We now choose an ordered basis $\mathcal{B}$ of $\mathfrak{a}$ as follows.
Pick ordered bases $\mathcal{B}_{ij}$ for each $\mathfrak{a}_{ij}$, order
the pairs $\left( i,j\right) $ lexicographically so that
\begin{equation}
\left( i,j\right) \succ \left( k,l\right) \text{ if }i>k\text{ or if }i=k%
\text{ and }j>l  \label{=lex0}
\end{equation}%
and let $\mathcal{B}$ be the \emph{descending} union of $\mathcal{B}_{ij}$.
Thus\emph{\ }$\mathcal{B}$ is ordered so that $\mathcal{B}_{ij}$ precedes $%
\mathcal{B}_{kl}$ if $\left( i,j\right) \succ \left( k,l\right) $. By the
Poincare-Birkhoff-Witt theorem the corresponding ordered (PBW) monomials
form a basis for the enveloping algebra $\mathcal{U}\left( \mathfrak{a}%
\right) $.

\begin{definition}
\label{contain}If $u\in \mathcal{U}\left( \mathfrak{a}\right) $ and a PBW
monomial $T$ has a nonzero coefficient in the expansion of $u$, we say $u$
\emph{contains} $T.$
\end{definition}

Note that $adX,adY$ act on $\mathcal{U}\left( \mathfrak{a}\right) $ with
integer eigenvalues as well and we define%
\begin{equation}
\mathcal{U}_{ij}\left( \mathfrak{a}\right) =\left\{ u\in \mathcal{U}\left(
\mathfrak{a}\right) :\left[ X,u\right] =iu,\left[ Y,u\right] =ju\right\}
\label{=Uij}
\end{equation}%
By construction each PBW monomial belongs to some $\mathcal{U}_{ij}\left(
\mathfrak{a}\right) $, and thus the following result is obvious.

\begin{lemma}
\label{lem-contain}If\ $u\in $ $\mathcal{U}_{ij}\left( \mathfrak{a}\right) $
and $u$ contains $T$, then $T\in $ $\mathcal{U}_{ij}\left( \mathfrak{a}%
\right) $.
\end{lemma}

\subsection{Coinvariants module} \label{coinv}

For $s\geq 0$ define $\mathcal{N}_{s}=\oplus _{i\geq s}\mathfrak{a}_{i}$;
then $\mathcal{N}_{0}$ is a Lie subalgebra and $\mathcal{N}_{1}$ is a
nilpotent ideal of $\mathcal{N}_{0}$. Let $\xi \in \mathcal{N}_{1}^{\ast }$
be such that $\xi |_{\mathcal{N}_{2}}=0$ then $\xi $ defines a Lie algebra
character of $\mathcal{N}_{1}$, and we have%
\begin{equation*}
\mathcal{N}_{0}^{\xi }=\mathfrak{a}_{0}^{\xi }\oplus \mathcal{N}_{1}
\end{equation*}%
where $\mathcal{N}_{0}^{\xi }$ and $\mathfrak{a}_{0}^{\xi }$ denote the
stabilizers of $\xi $ in the respective Lie algebras.

Consider the linear map $\DimaA{\Xi} :\mathcal{N}_{0}^{\xi }\rightarrow \mathcal{U}%
\left( \mathfrak{a}_{0}^{\xi }\right) $ given by
\begin{equation}
\DimaA{\Xi} \left( Z\right) =\left\{
\begin{tabular}{lll}
$Z$ & if & $Z\in \mathfrak{a}_{0}^{\xi }$ \\
$\xi \left( Z\right) $ & if & $Z\in \mathcal{N}_{1}$%
\end{tabular}%
\right.   \label{=Psi}
\end{equation}%
It is easy to check that $\DimaA{\Xi} $ is a Lie algebra map, i.e. it intertwines
the Lie bracket with the commutator, and hence by universality it extends to
an algebra map from $\mathcal{U}\left( \mathcal{N}_{0}^{\xi }\right) $ to $%
\mathcal{U}\left( \mathfrak{a}_{0}^{\xi }\right) $ that we continue to
denote by $\DimaA{\Xi} $.

Suppose $M$ is an $\mathcal{N}_{0}^{\xi }$-module, we define the $\xi $-coinvariant
module to be
\begin{equation*}
L=M/M^{\prime }\text{ where }M^{\prime }=span\left\{ Zv-\xi \left( Z\right)
v\mid Z\in \mathcal{N}_{1},v\in M\right\}
\end{equation*}%
Then $L$ is a $\mathfrak{a}_{0}^{\xi }$-module and the projection map $%
\varpi :M\rightarrow L$ satisfies
\begin{equation}
\varpi \left( uv\right) =\DimaA{\Xi} \left( u\right) \varpi \left( v\right) \text{
for }u\in \mathcal{U}\left( \mathcal{N}_{0}^{\xi }\right) ,v\in M\text{.}
\label{=inter}
\end{equation}

\subsection{The Key Lemma}\label{subsec:keylem}

In this subsection we assume the following.

\begin{cond}
\label{assume}
$ $
\begin{enumerate}
\item $\mathfrak{a}$ is a Lie algebra with elements $X,Y$ and bigrading $%
\mathfrak{a}_{ij}$ as in \S \ref{bigraded}.

\item $\mathfrak{a}_{ij}=\left\{ 0\right\} $ if $j\not\in \left\{
-1,0,1\right\} $ and also that $\mathfrak{a}_{1,-1}=\left\{ 0\right\} $.

\item $\xi $ is a character of $\mathcal{N}_{1}$ as in \S \ref{coinv}
and $\xi |_{\mathfrak{a}_{ij}}=0$ unless $i=1,j=0$.

\end{enumerate}
\end{cond}

\begin{lemma}
Suppose $\mathfrak{a}$ and $\xi $ satisfy (1)-(3) of Condition \ref{assume}
then we have $\mathfrak{a}_{0,1}\subset \mathfrak{a}_{0}^{\xi }.$
\end{lemma}

\begin{proof}
We need to show that $\DimaA{\xi} \left( \left[ A,B\right] \right) =0$ for all $%
A\in \mathfrak{a}_{0,1}$, $B\in \mathcal{N}_{1}$.

To prove this we may assume that $B\in \mathfrak{a}_{ij}$ for some $i,j$.
Then we have $\left[ A,B\right] \in \mathfrak{a}_{i,j+1}$ and so by
Condition \ref{assume} (3) we have $\DimaA{\xi} \left( \left[ A,B\right] \right) =0$
unless $i=1$ and $j=-1$. This forces $B\in $ $\mathfrak{a}_{1,-1}$ and hence
$B=0$ by Condition \ref{assume} (2) and so $\DimaA{\xi} \left( \left[ A,B\right]
\right) =0$ in this case as well.
\end{proof}

In view of the previous lemma we have a well-defined restriction map%
\begin{equation*}
Res:\left( \mathfrak{a}_{0}^{\DimaA{\xi} }\right) ^{\ast }\rightarrow \left(
\mathfrak{a}_{0,1}\right) ^{\ast }
\end{equation*}%

Let $M$ be a $\mathfrak{a}$-module and fix a (not necessarily good) filtration $F^iM$ on $M$.
We now define the $\DimaA{\xi} $-coinvariants module $L$ of $M$ as in \S \ref{coinv} and let $F^iL$ be the induced filtration on it.

Let
\begin{equation*}
AV_F\left( L\right) \subset \left( \mathfrak{a}_{0}^{\DimaA{\xi} }\right) ^{\ast }%
\text{ and }\mathcal{V}\left( M\right) \subset \mathfrak{a}^{\ast }
\end{equation*}%
denote the respective $F$-associated variety of $L$ and annihilator variety of $M$ as in
\S \ref{subsec:Filt} and \S \ref{subsubsec:AnnVar}.

We are now ready to formulate the key lemma.
\begin{lem}[\RamiA{The key lemma}]
\label{lem:key} Suppose $\phi \in Res\left[ AV_F\left( L\right) \right]
\subset \left( \mathfrak{a}_{0,1}\right) ^{\ast }$ and regard $\phi +\DimaA{\xi} $
as an element of $\mathfrak{a}^{\ast }$ via
\begin{equation*}
\left( \phi +\DimaA{\xi} \right) |_{\mathfrak{a}_{0,1}}=\phi \text{, }\left( \phi
+\DimaA{\xi} \right) |_{\mathfrak{a}_{1,0}}=\DimaA{\xi} \text{ and }\left( \phi +\DimaA{\xi}
\right) |_{\mathfrak{a}_{ij}}=0\text{ for all other pairs }\left( i,j\right)
\end{equation*}%
Then we have%
\begin{equation*}
\phi +\DimaA{\xi} \in \mathcal{V}\left( M\right)
\end{equation*}
\end{lem}

The proof involves in a crucial way the PBW basis discussed in \S \ref{bigraded} above. Note that by Condition \ref{assume} (2), the sequence of
pairs $\left( i,j\right) $ ordered as in (\ref{=lex0}) looks as follows:
\begin{equation}
\fbox{$\cdots ,\left( 1,1\right) $},\fbox{$\left( 1,0\right) ,(0,1)$},\fbox{$(0,0),(0,-1)$},\fbox{$\left( -1,1\right) ,\cdots $}  \label{=lex}
\end{equation}
where we have grouped the possible pairs $\left( i,j\right) $ into 4 groups
for ease of reference below. Note that we do not mean to imply that $%
\mathcal{B}_{ij}\neq \emptyset $ for the indicated pairs in (\ref{=lex}),
but rather that $\mathcal{B}_{ij}=\emptyset $ for the \emph{missing pairs}
e.g. $\left( 1,-1\right) ,\left( 0,2\right) $ etc.

\begin{proof}
[Proof of Lemma \ref{lem:key}]
\Dima{Let $\sigma^n:\mathcal{U}^n(\mathfrak{a}) \to \Sym^n(\mathfrak{a})$ denote the $n$-th symbol map. }
We need to show that for all $n$, and for all $
P\in Ann\left( M\right) \cap \mathcal{U}^{n}\left( \mathfrak{a}\right) $ we
have%
\begin{equation}
\left\langle \sigma ^{n}\left( P\right) ,\phi +\DimaA{\xi} \right\rangle =0
\label{=show}
\end{equation}
Since $Ann\left( M\right) $ and $\mathcal{U}^{n}\left( \mathfrak{a}\right) $
are stable under the adjoint action $ad$, $Ann\left( M\right) \cap \mathcal{U}^{n}\left( \mathfrak{a}\right) $ is a direct sum of $ad(X)$-eigenspaces.
Since $X$ any $Y$ commute, each $ad(X)$-eigenspace in $Ann\left( M\right) \cap \mathcal{U}^{n}\left( \mathfrak{a}\right) $ is a direct sum of $ad(Y)$-eigenspaces.
Thus we may further assume
\begin{equation*}
P\in Ann\left( M\right) \cap \mathcal{U}^{n}\left( \mathfrak{a}\right) \cap
\mathcal{U}_{kl}\left( \mathfrak{a}\right)
\end{equation*}%
for some integers $k,l$, where $\mathcal{U}_{kl}\left( \mathfrak{a}\right) $
is defined as in (\ref{=Uij}).

Consider the PBW monomials contained in $P$ in the sense of Definition \ref{contain}. We say such a monomial is \textquotedblleft
relevant\textquotedblright\ if it is a product of precisely $n$ factors from
group 2 in the sequence (\ref{=lex}) \Dima{(i.e. $\{(1,0),(0,1)\}$)} and \textquotedblleft
irrelevant\textquotedblright\ otherwise. Thus we get a decomposition
\begin{equation*}
P=R+I
\end{equation*}%
where $R$ and $I$ are combinations of relevant and irrelevant monomials
respectively.

We note that $R\in \mathcal{U}\left( \mathcal{N}_{0}^{\DimaA{\xi} }\right) $ and we
claim that the following properties hold
\begin{eqnarray}
\left\langle \sigma ^{n}\left( P\right) ,\phi +\DimaA{\xi} \right\rangle
&=&\left\langle \sigma ^{n}\left( R\right) ,\phi +\DimaA{\xi} \right\rangle
\label{C1} \\
\DimaA{\Xi} \left( R\right)  &\in &\mathcal{U}^{n-k}\left( \mathfrak{a}%
_{0,1}\right)   \label{C2} \\
\sigma ^{n-k}\left( \DimaA{\Xi} \left( R\right) \right)  &\in &Ann\left( Gr_F(L)\right)   \label{C3} \\
\left\langle \sigma ^{n}\left( R\right) ,\phi +\DimaA{\xi} \right\rangle
&=&\left\langle \sigma ^{n-k}\left( \DimaA{\Xi} \left( R\right) \right) ,\phi
\right\rangle   \label{C4}
\end{eqnarray}%
Granted these claims for the moment, we can prove the Lemma as
follows. Since $\phi \in Res\left[ AV_F\left( L\right) \right] $ we deduce
from (\ref{C2}) and (\ref{C3}) that $\left\langle \sigma ^{n-k}\left( \DimaA{\Xi}
\left( R\right) \right) ,\phi \right\rangle =0$. Now by (\ref{C1}) and (\ref%
{C4}) we get (\ref{=show}) as desired.

We now turn to the proof of claims (\ref{C1} -- \ref{C4}). For (\ref{C1})
its suffices to show that
\begin{equation}
\left\langle \sigma ^{n}\left( T\right) ,\phi +\DimaA{\xi} \right\rangle =0
\label{=showT}
\end{equation}%
for every \emph{irrelevant} monomial $T$ contained in $P$. Indeed if $T$ has
fewer then $n$ factors then $\sigma ^{n}\left( T\right) =0,$ otherwise $T$
must have a factor not from group 2 and then (\ref{=showT}) holds since
$\phi +\DimaA{\xi} $ vanishes on such factors by definition.

If $R=0$ then certainly (\ref{C2} -- \ref{C4}) hold. Therefore we may assume
that $P$ contains at least one relevant monomial $S$.  By definition every
such $S$ is of the form
\begin{equation*}
S=A_{1}\cdots A_{p}B_{1}\cdots B_{n-p}\text{ with }A_{i}\in \mathcal{B}_{1,0}%
\text{ and }B_{j}\in \mathcal{B}_{0,1}
\end{equation*}%
By Lemma \ref{lem-contain} we have $S\in \mathcal{U}_{kl}\left( \mathfrak{a}%
\right) $ which forces%
\begin{equation}
k,l\geq 0\text{ and }n=k+l.  \label{=kl}
\end{equation}%
and that $S$ is necessarily of the form%
\begin{equation}
S=A_{1}\cdots A_{k}B_{1}\cdots B_{n-k}\text{ with }A_{i}\in \mathcal{B}_{1,0}%
\text{ and }B_{j}\in \mathcal{B}_{0,1}  \label{=relev}
\end{equation}%
Now by (\ref{=Psi}) we get%
\begin{equation}
\DimaA{\Xi} \left( S\right) =\DimaA{\Xi} \left( A_{1}\cdots A_{k}B_{1}\cdots
B_{n-k}\right) =\DimaA{\xi} \left( A_{1}\right) \cdots \DimaA{\xi} \left( A_{k}\right)
B_{1}\cdots B_{n-k}\in \mathcal{U}^{n-k}\left( \mathfrak{a}_{0,1}\right)
\label{=cont}
\end{equation}%
Since $R$ is a combination of relevant monomials (\ref{C2}) follows.

To prove (\ref{C3}) we need to show that%
\begin{equation*}
\DimaA{\Xi} \left( R\right) L^{i}\subset L^{i+n-k-1}
\end{equation*}%
By formula (\ref{=inter}) we have
\begin{equation*}
\DimaA{\Xi} \left( R\right) L^{i}=\DimaA{\Xi} \left( R\right) \varpi \left( M^{i}\right)
=\varpi \left( RM^{i}\right) =\varpi \left( \left( P-I\right) M^{i}\right) .
\end{equation*}%
Since $P\in Ann\left( M\right) $ we have $PM^{i}=0$ and so it suffices to
show that
\begin{equation}
\varpi \left( TM^{i}\right) \subset L^{i+n-k-1}  \label{=showC3}
\end{equation}%
for every \emph{irrelevant} monomial $T$ contained in $P$. For this we
consider several cases.

First suppose $T$ has a group 1 factor, then we can write $T=ZT^{\prime }$
where $Z$ is a group 1 basis vector and $T^{\prime }$ is a smaller PBW
monomial. In this case we have $\DimaA{\xi} \left( Z\right) =0$ and hence we get
\begin{equation*}
\varpi \left( TM^{i}\right) =\varpi \left( ZT^{\prime }M^{i}\right) =\DimaA{\xi}
\left( Z\right) \varpi \left( T^{\prime }M^{i}\right) =0
\end{equation*}%
which certainly implies (\ref{=showC3}).

Thus we may suppose $T$ has no group 1 factors. It follows then that the
only possible factors of $T$ with positive $ad$ $X$ weight are those from $%
\mathcal{B}_{1,0}$. Now suppose that $T$ has a group 4 factor. Since such a
factor has negative $ad$ $X$ weight and so since $T$ has $ad$ $X$ weight $k$%
, $T$ must have\ at least $k+1$ factors from $\mathcal{B}_{1,0}$. Thus $%
T=A_{1}\cdots A_{k+1}T^{\prime }$ where $A_{i}\in \mathcal{B}_{1,0}$ and $%
T^{\prime }\in \mathcal{U}^{n-k-1}\left( \mathfrak{a}\right) $. Thus we get
\begin{equation*}
\varpi \left( TM^{i}\right) =\DimaA{\xi} \left( A_{1}\right) \cdots \DimaA{\xi} \left(
A_{k+1}\right) \varpi \left( T^{\prime }M^{i}\right) \subset L^{i+n-k-1}
\end{equation*}

Therefore we may assume that $T$ has only group 2 and group 3 factors. Since
$T\in \mathcal{U}_{kl}\left( \mathfrak{a}\right) $ it follows that $T$ must
have exactly $k$ factors from $\mathcal{B}_{1,0}$ and \emph{at least }$l$
factors from $\mathcal{B}_{0,1}$. Since $T$ has at most $n$ factors and $%
k+l=n,$ it follows that $T$ has \emph{exactly} $l$ factors from $\mathcal{B}%
_{0,1}$. Hence $T$ is relevant, contrary to assumption. This finishes the
proof of (\ref{=showC3}).

Finally to prove (\ref{C4}) it suffices to show that%
\begin{equation*}
\left\langle \sigma ^{n}\left( S\right) ,\phi +\DimaA{\xi} \right\rangle
=\left\langle \sigma ^{n-k}\left( \DimaA{\Xi} \left( S\right) \right) ,\phi
\right\rangle
\end{equation*}%
for every relevant monomial $S=A_{1}\cdots A_{k}B_{1}\cdots B_{n-k}$  as in (%
\ref{=relev}). For this we calculate as follows%
\begin{equation*}
\left\langle \sigma ^{n}\left( S\right) ,\phi +\DimaA{\xi} \right\rangle
=\left\langle \DimaA{\xi} \left( A_{1}\right) \cdots \DimaA{\xi} \left( A_{k}\right)
\sigma ^{n-k}\left( B_{1}\cdots B_{n-k}\right) ,\phi \right\rangle
=\left\langle \sigma ^{n-k}\left( \DimaA{\Xi} \left( S\right) \right) ,\phi
\right\rangle
\end{equation*}
\end{proof}

\subsection{Proof of Theorem \ref{thm:HCAdm}}\label{subsec:PfHCAdm}

In the notations of Theorem \ref{thm:HCAdm}:

It is enough to show that $\oE^d(\pi)$ is admissible.
Let $(\tau,L):= \oE^d(\pi)= \pi_{\fu^{d-1},\psi},$ considered as a representation of $\fp_{n-d+1} = \g_{n-d} \ltimes \fv_{n-d+1}$.
Denote by $\varpi$ the projection $\pi \onto \tau$.

We will need the following lemma from linear algebra.
\begin{lem}\label{lem:LinAlg}
Let $u \in \fu_n^d$ be the matrix that has 1s on the superdiagonal of the lower block and 0s elsewhere. Then for any $v \in \fv_{n-d+1}$, if $(u+v)^d=0$ then $v=0$.
\end{lem}
\begin{proof}
Let $A:=u+v$. Computing $A^k$ by induction for $k \leq d$ we see that its first $n-d$ columns will be zero, the $n-d+k$-th column of the submatrix consisting of the first $n-d$ rows will be $v$, the other columns of this submatrix will be zero and the square submatrix formed by the last $d$ rows and columns will be $J_d^k$, where $J_d$ is the (upper triangular) Jordan block. Thus, $A^d=0$ if and only if $v=0$.
\end{proof}

Fix a good filtration $\pi^i$ on $\pi$. Note that by Proposition \ref{prop:WonFilt} it will also be good as an $\n$-filtration and define $L^i:=\varpi(\pi^i)$.  Note that $L^i$ is a good filtration.

\begin{cor} \label{cor:AV0}
$pr_{\fv_{n-d+1}^*}(AV(L)) = \{0\}$.
\end{cor}
\begin{proof}
Let $\mathfrak{a}:=\g_n$ and let $X,Y \in \g$ be diagonal matrices given by $$X = \diag(0^{n-d},1,2 , \dots ,  d) \text{ and }Y= \diag(0^{n-d-1},1^{d+1}).$$ Consider the bigrading $\mathfrak{a}=\bigoplus_{ij} \mathfrak{a}_{ij}$ defined as in \S \ref{bigraded}. Note that $\cN_1 = \DimaA{\fu_d}$ and let $\psi=\psi^d$. Note that the conditions of Condition \ref{assume} are satisfied.

Let $\phi \in pr_{\fv_{n-d+1}^*}(AV(L))$. By the  key Lemma \ref{lem:key}, we have $\phi+\psi \in \cV(\pi)$. By the definition of $d$, this implies $((\phi+\psi)^*)^d=0$ and thus, by Lemma
\ref{lem:LinAlg}, $\phi=0$.
\end{proof}

%

\begin{proof}[Proof of Theorem \ref{thm:HCAdm}]
By Corollary \ref{cor:AV0}, $pr_{\fv_{n-d+1}^*}(AV(L)) = \{0\}$.
Hence any $X\in \fv_{n-d+1}$ vanishes on $AV(L) \subset \p_{n-d+1}^*$.
By Hilbert's Nullstellensatz this implies that there exists $k$ such that $X^k \in Ann(Gr(L))$. Since $\fv_{n-d+1}$ is finite dimensional, one can find one $k$ suitable for all $X\in \fv_{n-d+1}$. Since $L^i$ is an $\n_{n-d+1}$-good filtration on $L$, $Gr(L)$ is finitely generated over $\Sym(\n_{n-d+1})$. Since $\Sym^{>k}(\fv_{n-d+1})$ acts by zero, $Gr(L)$ is finitely generated even over $\Sym(\n_{n-d})$. Thus, $L^i$ is an $\n_{n-d}$- good filtration and hence $L$ is finitely generated over $\n_{n-d}$. Thus, by Theorem \ref{thm:EqAdm} it is an admissible Harish-Chandra module over $G_{n-d}$.
\end{proof}

\appendix

\section{Proof of Proposition \ref{prop:WonFilt}} \label{sec:PfWonFilt}
\setcounter{lemma}{0}
The current formulation of the proposition and the proof written down in this appendix are results of a conversation of one of the authors with J. Bernstein.
After writing it down we learnt that Proposition \ref{prop:WonFilt} was proven by O. Gabber and written down by A. Joseph in his lecture notes \cite{Jos81}.
We decided to keep this appendix for the convenience of the reader.

We will need some lemmas on filtrations.

\begin{defn}
Two filtrations $F$ and $\Phi$ are called comparable if there exists $k >0$
such that for all $i \geq 0$ we have $\Phi^{i}M \subset F^{i+k}M \subset\Phi^{i+2k}M$.
\end{defn}

\begin{lem} \label{lem:ComparFilt}
Suppose $F$ and $\Phi$ are comparable. Then there exist filtrations $\Psi_j$, $0 \leq j \leq 2k$ such that $\Psi_0^iM = F^{i+k}M$,
$\Psi_{2k}^iM = \Phi^{i+2k}M$, and $\Psi^i_jM \subset \Psi^i_{j+1}M \subset \Psi^{i+1}_jM \subset \Psi^{i+1}_{j+1}M$.
\end{lem}
\begin{proof}
Define $\Psi^i_jM:= F^{i+k}M + \Phi^{i+j}M$.
\end{proof}

\begin{lem}\label{lem:GoodCompar}
Any two good filtrations are comparable.
\end{lem}

\begin{cor}\label{cor:AVGood}
Let $\Phi$ and $\Psi$ be 2 good filtrations on $M$. Then $\operatorname{Gr}%
_{\Phi}(M)$ and $\operatorname{Gr}_{\Psi}(M)$ are Jordan-Holder equivalent. In
particular,  $AV_{\Phi}(M) $ does not depend on the choice of
good filtration $\Phi$.
\end{cor}

\begin{lem}
Let $F$ and $\Phi$ be two filtrations on $M$ such that $F^iM \subset \Phi^iM \subset F^{i+1}M \subset \Phi^{i+1}M$ for any $i\geq 0$.
Consider the natural morphism $\phi: \Gr_F(M) \to \Gr_{\Phi}(M)$ and let $K:=\Ker \phi$, and $C:=\Coker \phi$. Then $C$ is isomorphic to the graded $\Gr(A)$-module $K[1]$  obtained from $K$ by shifting the gradation by one.
\end{lem}
\begin{proof}
The $i$-th grade of $\phi$ is the natural morphism $F^iM /F^{i-1}M \to  \Phi^iM/\Phi^{i-1}M$. Therefore $K_i = \Phi^{i-1}M/F^{i-1}M$ and $C_i=\Phi^{i}M/F^{i}M$.
\end{proof}

\begin{cor} \label{cor:CompGoodFilt}
Suppose $\Gr(A)$ is Noetherian. Let $F$ and $\Phi$ be two comparable filtrations on $M$. Then $F$ is a good filtration if and only if $\Phi$ is a good filtration.
\end{cor}

\begin{proof}
By Lemma \ref{lem:ComparFilt}, we may suppose that $F^iM \subset \Phi^iM \subset F^{i+1}M \subset \Phi^{i+1}M$ for any $i\geq 0$.
Recall that a filtration is good if and only if the associated graded module is finitely generated.
Consider the exact sequence of $\Gr(A)$-modules
$0 \to K \to \Gr_F(M) \to \Gr_{\Phi}(M) \to C$ from the previous lemma. Suppose that $F$ is a good filtration. Then $\Gr_F(M)$ is finitely generated and hence Noetherian. Thus so is $K$, and by the previous lemma so is $C$. Hence $\Gr_{\Phi}(M)$ is also finitely generated and hence $\Phi$ is good. The other implication is proven in a similar way.
\end{proof}

The proof of Proposition \ref{prop:WonFilt} will be based on the following proposition from \cite[ \S 3.7]{Wal1}.

\begin{prop}[Casselman-Osborne]
\label{prop:Osborne} Let $\mathfrak{n}$ denote the nilradical of the
complexified Lie algebra of the minimal parabolic subgroup of $G$ and
${\mathfrak{k}}$ denote the complexified Lie algebra of the maximal compact
subgroup of $G$. Let $Z_{G}({\mathfrak{g}}):=U({\mathfrak{g}})^{G}$. Then

There exists a finite-dimensional subspace $E \subset U({\mathfrak{g}})$
such that
\[
U({\mathfrak{g}})^{i} \subset U(\mathfrak{n})^{i}EZ_{G}({\mathfrak{g}%
})U({\mathfrak{k}}) \text{ and therefore } U({\mathfrak{g}})= U(\mathfrak{n}%
)EZ_{G}({\mathfrak{g}})U({\mathfrak{k}})
\]
\end{prop}

\begin{proof}[Proof of Proposition \ref{prop:WonFilt}]
Let $\pi \in \cM_{HC}(G)$. Let us first construct one $\n$-good $\g$-filtration.

Let $V \subset \pi$ be a finite dimensional $K$-invariant and $Z_{G}$ -invariant generating subspace. Let $E \subset U({\mathfrak{g}})$ be such that $U({\mathfrak{g}})^{i} \subset U(\mathfrak{n})^{i}EZ_{G}({\mathfrak{g}
})U({\mathfrak{k}}), $ as in Proposition \ref{prop:Osborne}.  Define $F^i :=U(\g)^{i} V$ and $\Phi^i:= U(\mathfrak{n})^{i}E V= U(\mathfrak{n})^{i}EZ_{G}({\mathfrak{g}
})U({\mathfrak{k}})V$. Note that $\Phi^i$ is a good $\n$-filtration and that $F^i \subset \Phi^i \subset F^{i+k}$, where $k$ is the maximal degree of an element in $E$.
Thus, by Corollary \ref{cor:CompGoodFilt}, $F^i$ is a good $\n$-filtration.

Now, let $f^i$ be any good $\g$-filtration on $\pi$. By Lemma \ref{lem:GoodCompar} $f^i$ is comparable to $F^i$ and thus, by Corollary \ref{cor:CompGoodFilt}, $f^i$ is $\n$-good.
\end{proof}

\end{document}